\documentclass[11pt]{amsart}
\usepackage{amsfonts}
\usepackage{amsmath}
\usepackage{amssymb}
\usepackage{amsthm}
\usepackage[margin=3truecm]{geometry}
\usepackage[all]{xy}

\newtheorem{Theorem}{Theorem}[section]

\newtheorem{Lemma}[Theorem]{Lemma}
\newtheorem{Proposition}[Theorem]{Proposition}

\theoremstyle{remark}
\newtheorem{Remark}[Theorem]{Remark}

\theoremstyle{definition}
\newtheorem{Definition}[Theorem]{Definition}

\newtheorem{Question}[Theorem]{Question}

\begin{document}
\bibliographystyle{plain}
\title{Sobolev-Lorentz capacity and its regularity in the Euclidean setting}
\author[\c{S}. Costea]{\c{S}erban Costea}
\address{\c{S}. Costea\\
Department of Mathematics and Computer Science\\
University of Pite\c sti\\
Str. T\^argul din Vale nr. 1\\
RO-110040 Pite\c sti, Arge\c s, Romania}
\email{serban.costea@upit.ro, secostea@hotmail.com}


\thanks{The author was partly supported by the University of Pisa via grant PRA-2015-0017.}

\begin{abstract}
This paper studies the Sobolev-Lorentz capacity and its regularity in the Euclidean setting for $n \ge 1$ integer. We extend here our previous results on the Sobolev-Lorentz capacity obtained for $n \ge 2.$

Moreover, for $n \ge 2$ integer we obtain a few new results concerning the $n,1$ relative and global capacities. Specifically, we obtain sharp estimates for the $n,1$ relative capacity of the concentric condensers $(\overline{B}(0,r), B(0,1))$ for all $r$ in $[0,1).$ As a consequence we obtain the exact value of the $n,1$ capacity of a point relative to all its bounded open neighborhoods from ${\mathbf{R}}^n$ when $n \ge 2.$ These new sharp estimates concerning the $n,1$ relative capacity improve some of our previous results. We also obtain a new result concerning the $n,1$ global capacity. Namely, we show that this aforementioned constant is also the value of the $n,1$ global capacity of any point from ${\mathbf{R}}^n,$ where $n \ge 2$ is integer.

Computing the aforementioned exact value of the $n,1$ relative capacity of a point with respect to all its bounded open neighborhoods from ${\mathbf{R}}^n$ allows us to give a new prove of the embedding $H_{0}^{1,(n,1)}(\Omega) \hookrightarrow C(\overline{\Omega}) \cap L^{\infty}(\Omega),$ where $\Omega \subset {\mathbf{R}}^n$ is open and $n \ge 2$ is an integer.

In the penultimate section of our paper we prove a new weak convergence result for bounded sequences in the non-reflexive spaces $H^{1,(p,1)}(\Omega)$ and $H_{0}^{1,(p,1)}(\Omega).$ The weak convergence result concerning the spaces $H^{1,(p,1)}(\Omega)$ is valid whenever $1<p<\infty,$ while the weak convergence result concerning the spaces $H_{0}^{1,(p,1)}(\Omega)$ is valid whenever $1 \le n<p<\infty$ or $1<n=p<\infty.$

As a consequence of the weak convergence result concerning the spaces $H_{0}^{1,(p,1)}(\Omega),$ in the last section of our paper we show that the relative and the global $(p,1)$ and $p,1$ capacities are Choquet whenever $1 \le n<p<\infty$ or $1<n=p<\infty.$
\end{abstract}

\maketitle
\section{Introduction}
In this paper we study the Sobolev-Lorentz capacity and its regularity in the Euclidean setting for $n \ge 1.$ This paper is motivated by the work of Stein-Weiss \cite{SW} and Bennett-Sharpley \cite{BS}
on Lorentz spaces and by the work of Stein \cite{Ste} and Cianchi-Pick \cite{CiPi1}, \cite{CiPi2} on Sobolev-Lorentz spaces.

We studied the Sobolev-Lorentz spaces and their associated capacities extensively in our previous work. In this paper we extend some of the previous results obtained in our book \cite{Cos3} and in our papers \cite{Cos1} and \cite{Cos4}. In \cite{Cos3} we studied the Sobolev-Lorentz spaces and the associated Sobolev-Lorentz capacities in the Euclidean setting for $n \ge 2.$ The restriction on $n$ there as well as in \cite{Cos1} was due to the fact that we studied the
$n,q$ capacity for $n>1.$ In our recent paper \cite{Cos4} we studied the Sobolev-Lorentz spaces in the Euclidean setting for $n \ge 1.$ There we extended to the case $n=1$ many of the results on Sobolev-Lorentz spaces obtained in \cite{Cos0} and \cite{Cos3} for $n \ge 2.$ In this paper we extend to $n=1$ many on the results on Sobolev-Lorentz capacities obtained in \cite{Cos0} and \cite{Cos3} for $n \ge 2.$

The Lorentz spaces were studied by Bennett-Sharpley in \cite{BS} and by Stein-Weiss in \cite{SW}.

The Sobolev-Lorentz spaces have also been studied by Stein in \cite{Ste}, Cianchi-Pick in \cite{CiPi1} and \cite{CiPi2}, by Kauhanen-Koskela-Mal{\'{y}} in \cite{KKM}, and by Mal{\'{y}}-Swanson-Ziemer in \cite{MSZ}. We studied the Sobolev-Lorentz relative $p,q$-capacity in the Euclidean setting (see \cite{Cos0}, \cite{Cos1} and \cite{Cos3}). See also our joint work \cite{CosMaz} with V. Maz'ya.

The classical Sobolev spaces were studied by Gilbarg-Trudinger in \cite{GT}, Maz'ja in \cite{Maz}, Evans in \cite{Eva}, Heinonen-Kilpel\"{a}inen-Martio in \cite{HKM}, and by Ziemer in \cite{Zie}. The Sobolev $p$-capacity was studied by Maz'ya \cite{Maz} and by Heinonen-Kilpel\"{a}inen-Martio \cite{HKM} in ${\mathbf{R}}^n.$

After recalling the definition of Lorentz spaces and some of its basic properties in Section \ref{section Lorentz spaces}, we move to Section \ref{section Sobolev-Lorentz spaces}, where we recall the definition of the Sobolev-Lorentz spaces and some of the results that are to be used later in the paper.

In Section \ref{section Sobolev-Lorentz capacity} we study the basic properties of the Sobolev-Lorentz capacities on ${\mathbf{R}}^n$ for $n \ge 1.$ There we study the global Sobolev-Lorentz capacities ${\rm{Cap}}_{(p,q)}(\cdot)$ and ${\rm{Cap}}_{p,q}(\cdot)$ and the relative Sobolev-Lorentz capacities ${\rm{cap}}_{(p,q)}(\cdot, \Omega)$ and ${\rm{cap}}_{p,q}(\cdot, \Omega)$ for $\Omega \subset {\mathbf{R}}^n$ bounded and open, $n \ge 1$ integer, $1<p<\infty$ and $1 \le q \le \infty.$ The $p,q$-capacity is associated to the Lorentz $p,q$-quasinorm while the $(p,q)$-capacity is associated to the Lorentz $(p,q)$-quasinorm. The case $p=q$ yields the $p$-capacity, studied extensively in literature.

In Section \ref{section Sobolev-Lorentz capacity} we revisit many of the basic properties of the Sobolev-Lorentz capacities, studied extensively in Chapter 4 of our book \cite{Cos3} for $n \ge 2$ and we extend them to the case $n=1.$ The results that we extend here concern the monotonicity, the convergence, the countable subadditivity and the regularity of these capacities. The regularity of these capacities was extended in this section to the case $n=1$ for $1<q<\infty$ when we worked with the $(p,q)$ global and the $(p,q)$ relative capacities and for $1<q<p$ when we worked with the $p,q$ global and the $p,q$ relative capacities.

Due to the non-reflexivity of the Sobolev-Lorentz spaces $H_{0}^{1, (p,1)}(\Omega)$ and $H_{0}^{1,(p,\infty)}(\Omega),$ it is challenging to prove the Choquet property for the corresponding relative and global capacities associated to these non-reflexive Sobolev-Lorentz spaces.
Also, due to the fact that the $p,q$-quasinorm is not a norm when $p<q \le \infty,$ the Choquet property of the $p,q$ relative and global capacities is not known when $q$ is in the range $(p, \infty].$

No positive results on the Choquet property for the corresponding relative and global capacities associated to these non-reflexive Sobolev-Lorentz spaces have been obtained until now. In this paper we obtain a few partial positive new results concerning the Choquet property of $(p,1)$ and the $p,1$ relative and global capacities. Namely, in Section \ref{section Choquet property for the p1 capacities} we show that the global Sobolev-Lorentz capacities ${\rm{Cap}}_{(p,1)}(\cdot)$ and ${\rm{Cap}}_{p,1}(\cdot)$ as well as the relative Sobolev-Lorentz capacities ${\rm{cap}}_{(p,1)}(\cdot, \Omega)$ and ${\rm{cap}}_{p,1}(\cdot, \Omega)$ are Choquet whenever $1 \le n<p<\infty$ or $1<n=p<\infty.$ Here $\Omega \subset {\mathbf{R}}^n$ is a bounded and open set and $n \ge 1$ is an integer. See Theorems \ref{Choquet (p,1) relative capacity Thm} and \ref{Choquet p,1 relative capacity Thm} for the regularity of the relative capacities. Theorems \ref{Choquet (p,1) global capacity Thm} and \ref{Choquet p,1 global capacity Thm} deal with the regularity of the global capacities.

In order to prove the regularity of these capacities we needed to prove a Monotone Convergence Theorem for each of them. See Theorem \ref{MCT for the (p,1) and p,1 relative capacities} for the relative capacities and Theorem \ref{MCT for the (p,1) and p,1 global capacities} for the global capacities. These are new results as well.

When proving the Choquet property of the $(p,1)$ and $p,1$ relative and global capacities for these values of $n$ and $p$ (that is, $1 \le n<p<\infty$ or $1<n=p<\infty$) we used many times the fact that for these values of $n$ and $p$ we can work with continuous admissible functions from $H_{0}^{1,(p,1)}(\Omega).$

Indeed, in \cite{Cos4} we proved that the spaces $H_{0}^{1,(p,q)}(\Omega)$ embed into the space $C^{1-\frac{n}{p}}(\overline{\Omega})$ of H\"{o}lder continuous functions on $\overline{\Omega}$ with exponent $1-\frac{n}{p}$ whenever $\Omega \subset {\mathbf{R}}^n$ is open and $1 \le n, q \le \infty.$
See \cite[Theorem 5.5 (iii)]{Cos4} for $1=n<p<\infty$ and \cite[Theorem 5.6 (iv)]{Cos4} for $1<n<p<\infty.$

For $1<n=p<\infty$ we give a new prove of the embedding $H_{0}^{1,(n,1)}(\Omega) \hookrightarrow C(\overline{\Omega}) \cap L^{\infty}(\Omega)$ in section \ref{section Sharp estimates of the Sobolev-Lorentz n1 relative capacity} and we find the optimal constant for the embedding.
See Theorem \ref{continuous embedding of H01n1 into C cap Linfty} (ii).

This embedding was obtained by Stein in his paper \cite{Ste} and by Cianchi-Pick
(see \cite[Theorem 3.5 (i)]{CiPi1}) with the same optimal constant that we obtained in this paper.

Our proof of this embedding is different. We use a new approach. Specifically, we use the theory of the $n,1$ relative capacity in ${\mathbf{R}}^n, n \ge 2.$ In Section \ref{section Sharp estimates of the Sobolev-Lorentz n1 relative capacity} we obtain a new result by improving our estimates from \cite[Theorem 3.11]{Cos1} for the $n,1$ relative capacity of the condensers $(\overline{B}(0,r), B(0,1))$ and extending them to ALL $r$ in $[0,1).$ See Theorem \ref{sharp estimates for the n1 relative capacity of a point} (i). In particular, we obtain the exact value for the $n,1$ capacity of a point relative to all its bounded open neighborhoods from ${\mathbf{R}}^n,$ a strictly positive number as we saw in \cite[Corollary 3.8]{Cos1}. See Theorem \ref{sharp estimates for the n1 relative capacity of a point} (ii). Moreover, in this section we obtain a new result for the global $n,1$ capacity as well. Namely, in Theorem \ref{sharp estimate for the n1 global capacity of a point} we show that the value from Theorem \ref{sharp estimates for the n1 relative capacity of a point} (ii) is also the value of the global $n,1$ capacity of any point from ${\mathbf{R}}^n.$

By using the theory of the $n,1$ relative capacity in ${\mathbf{R}}^n, n \ge 2,$ we see that this aforementioned constant shows up in the embedding $H_{0}^{1,(n,1)}(\Omega) \hookrightarrow C(\overline{\Omega}) \cap L^{\infty}(\Omega).$ See (\ref{Linfty norm of u le Ln1 norm of nabla u}).

Thus, Section \ref{section Sharp estimates of the Sobolev-Lorentz n1 relative capacity} together with our paper \cite{Cos4} (see \cite[Theorems 3.5, 4.3, 4.13 and 5.6]{Cos4}) reinforce the fact that for every $n \ge 2$ integer and for every $\Omega \subset {\mathbf{R}}^n$ open, the space $H_{loc}^{1,(n,1)}(\Omega)$ is the largest Sobolev-Lorentz space defined on $\Omega$ for which each function has a version in $C(\Omega).$

This embedding result from Section \ref{section Sharp estimates of the Sobolev-Lorentz n1 relative capacity} is being relied on heavily in Section \ref{section Bounded sequences in non-reflexive Sobolev-Lorentz spaces} and in Section \ref{section Choquet property for the p1 capacities}.

In Section \ref{section Bounded sequences in non-reflexive Sobolev-Lorentz spaces} we prove a new weak convergence result for bounded sequences in the non-reflexive spaces $H^{1,(p,1)}(\Omega)$ and $H_{0}^{1,(p,1)}(\Omega).$ See Theorem \ref{Bdd in H01p1 weak limit in H01p1}.

This new weak convergence result concerning $H^{1,(p,1)}(\Omega)$ holds for all $p$ in $(1,\infty)$ and for all integers $n \ge 1.$ See Theorem \ref{Bdd in H01p1 weak limit in H01p1} (i). We fix $q$ in $(1,\infty).$ We show that even in a non-reflexive space such as $H^{1,(p,1)}(\Omega),$ if we have a bounded sequence $u_k$ in $H^{1,(p,1)}(\Omega)$ such that $(u_k, \nabla u_k)$ converges weakly to $(u, \nabla u)$ in $L^{(p,q)}(\Omega) \times L^{(p,q)}(\Omega;{\mathbf{R}}^n),$ then the function $u$ is in the reflexive space $H^{1,(p,s)}(\Omega)$ and in fact $(u_k, \nabla u_k)$ converges weakly to $(u, \nabla u)$ in $L^{(p,s)}(\Omega) \times L^{(p,s)}(\Omega;{\mathbf{R}}^n)$ whenever $1<s<\infty.$

Then we show that this limit function $u$ is also in the non-reflexive space $H^{1,(p,1)}(\Omega).$ This
task is challenging to prove. Due to the non-reflexivity of the spaces $L^{p,1}(\Omega; {\mathbf{R}}^m),$
we do not know whether the sequence $(u_k, \nabla u_k)$ converges weakly to $(u, \nabla u)$ in $L^{p,1}(\Omega) \times L^{p,1}(\Omega;{\mathbf{R}}^n)$ or not. Although we cannot rely on the weak-$*$ lower semicontinuity of the $p,1$ norm, we manage to prove the membership of $u$ in $H^{1,(p,1)}(\Omega)$ and a Fatou-type result for $u$ and for $\nabla u$ with respect to both the $p,1$ norm and the $(p,1)$ norm.

The new weak convergence result for $H_{0}^{1,(p,1)}(\Omega)$ is even more challenging to prove. We managed to prove it for $1 \le n<p<\infty$ and for $1<n=p<\infty.$ See Theorem \ref{Bdd in H01p1 weak limit in H01p1} (ii). When proving this weak convergence result for $H_{0}^{1,(p,1)}(\Omega)$ we relied heavily many times on the fact that for these values of $n$ and $p$ we can work with continuous functions from $H_{0}^{1,(p,1)}(\Omega).$ This leaves for instance as an open question the membership of the limit function $u$ in $H_{0}^{1,(p,1)}(\Omega)$ when $1<p<n,$ $\Omega \subset {\mathbf{R}}^n$ is bounded and $u$ is not compactly supported in $\Omega.$

This new weak convergence theorem from Section \ref{section Bounded sequences in non-reflexive Sobolev-Lorentz spaces} is being put to use later in Section \ref{section Choquet property for the p1 capacities} to prove the Choquet property of the global Sobolev-Lorentz capacities ${\rm{Cap}}_{(p,1)}(\cdot)$ and ${\rm{Cap}}_{p,1}(\cdot)$ and of the relative Sobolev-Lorentz capacities ${\rm{cap}}_{(p,1)}(\cdot, \Omega)$ and ${\rm{cap}}_{p,1}(\cdot, \Omega)$ whenever $1 \le n<p<\infty$ or $1<n=p<\infty.$ Like before, $\Omega \subset {\mathbf{R}}^n$ is a bounded and open set and $n \ge 1$ is an integer. The existence of discontinuous and/or unbounded functions in $H_{0}^{1,(p,1)}(\Omega)$ when $1<p<n$ prevents us for now from extending the Choquet property of the $p,1$ and $(p,1)$ relative and global capacities to the case $1<p<n.$

\section{Notations}
\label{section Notations}

 Here we recall the standard notation to be used throughout this paper. (See also \cite{Cos4}).
 Throughout this paper, $C$ will denote a positive constant whose value is not necessarily the same at
 each occurrence; it may vary even within a line. $C(a,b, \ldots)$ is a constant that depends only on the parameters $a,b, \cdots.$

 Throughout this paper $\Omega$ will denote a nonempty open subset of ${\mathbf{R}}^n,$ while $dx=d m_n(x)$ will denote the Lebesgue $n$-measure in ${\mathbf{R}}^n,$ where $n \ge 1$ is an integer.
 For $E \subset {\mathbf{R}}^n,$ the boundary, the closure, and the complement of $E$ with
 respect to ${\mathbf{R}}^n$ will be denoted by $\partial E,$ $\overline{E},$ and
 $\mathbf{R}^n \setminus E,$ respectively, while $|E|=\int_{E} dx$ will denote the Lebesgue
 measure of $E$ whenever $E$ is measurable; $E \subset \subset F$ means that $\overline{E}$
 is a compact subset of $F.$

 Moreover, $B(a,r)= \{ x \in {\mathbf{R}}^n: |x-a|<r \}$ is the open ball with center
 $a \in {\mathbf{R}}^n$ and radius $r>0,$ while $\overline{B}(a,r)= \{ x \in {\mathbf{R}}^n: |x-a| \le r \}$ is the closed ball with center $a \in \mathbf{R}^n$ and radius $r>0.$

 For $n \ge 1$ integer, $\Omega_n$ denotes the Lebesgue measure of the
 $n$-dimensional unit ball. (That is, $\Omega_n=|B(0,1)|$). For $n \ge 2$ integer,
 $\omega_{n-1}$ denotes the spherical measure of the $n-1$-dimensional sphere;
 thus, $\omega_{n-1}=n \Omega_n$ for every integer $n \ge 2.$

 For a Lebesgue measurable function $u: \Omega \rightarrow {\mathbf{R}},$ $\mbox {supp } u$ is the
 smallest closed set such that $u$ vanishes outside $\mbox {supp } u.$

 For a Lebesgue measurable vector-valued function $f=(f_1, \ldots, f_m):
 \Omega \rightarrow \mathbf{R}^m,$ we let
 \begin{equation*}\label{def abs val of a vector function}
 |f|=\sqrt{f_1^2+f_2^2+\ldots+f_m^2}.
 \end{equation*}

\section{Lorentz spaces}
\label{section Lorentz spaces}

For the next three subsections we follow mostly our paper \cite{Cos4}.
\subsection{Definitions and basic properties}
 Let $f:\Omega \rightarrow \mathbf{R}$ be a measurable function. We
 define $\lambda_{[f]},$ the \textit{distribution function} of $f$ as follows (see Bennett-Sharpley \cite[Definition II.1.1]{BS} and Stein-Weiss \cite[p.\ 57]{SW}):
 $$\lambda_{[f]}(t)=|\{x \in \Omega: |f(x)| > t \}|, \qquad t \ge 0.$$
 We define $f^{*},$  the \textit{nonincreasing rearrangement} of $f$ by
 $$f^{*}(t)=\inf\{v: \lambda_{[f]}(v) \le t \}, \quad t \ge 0.$$
 (See Bennett-Sharpley \cite[Definition II.1.5]{BS} and Stein-Weiss \cite[p.\ 189]{SW}).
 We notice that $f$ and $f^{*}$ have the same distribution function.
 Moreover, for every positive $\alpha$ we have $(|f|^{\alpha})^{*}=(|f|^{*})^{\alpha}$
 and if $|g|\le |f|$ a.e. on $\Omega,$ then $g^{*}\le f^{*}.$
 (See Bennett-Sharpley \cite[Proposition II.1.7]{BS}).
 We also define $f^{**}$, the \textit{maximal function} of $f^{*}$ by
 $$f^{**}(t)=m_{f^{*}}(t)=\frac{1}{t} \int_{0}^{t} f^{*}(s) ds, \quad t >0.$$
 (See Bennett-Sharpley \cite[Definition II.3.1]{BS} and Stein-Weiss \cite[p.\ 203]{SW}).

Throughout this paper, we denote by $q'$ the H\"{o}lder
conjugate of $q \in [1,\infty].$

The \textit{Lorentz space} $L^{p,q}(\Omega),$ $1<p<\infty,$ $1\le
q\le \infty,$ is defined as follows:
$$L^{p,q}(\Omega)= \{f: \Omega \rightarrow \mathbf{R}: f \mbox { is measurable and }
||f||_{L^{p,q}(\Omega)}<\infty\},$$
where
$$||f||_{L^{p,q}(\Omega)}=||f||_{p,q}=\left\{ \begin{array}{lc}
\left( \int_{0}^{\infty} (t^{\frac{1}{p}}f^{*}(t))^q \, \frac{dt}{t}
\right)^{\frac{1}{q}} & 1 \le q < \infty \\
\sup_{t>0} t \lambda_{[f]}(t)^{\frac{1}{p}}=\sup_{s>0}
s^{\frac{1}{p}} f^{*}(s) & q=\infty.
\end{array}
\right.
$$
(See Bennett-Sharpley \cite[Definition IV.4.1]{BS} and Stein-Weiss \cite[p.\ 191]{SW}). If $1 \le
q\le p,$ then $||\cdot||_{L^{p,q}(\Omega)}$ already represents a
norm, but for $p < q \le \infty$ it represents a quasinorm that is
equivalent to the norm $||\cdot||_{L^{(p,q)}(\Omega)},$ where
$$||f||_{L^{(p,q)}(\Omega)}=||f||_{(p,q)}=\left\{ \begin{array}{lc}
\left( \int_{0}^{\infty} (t^{\frac{1}{p}}f^{**}(t))^q \, \frac{dt}{t} \right)^{\frac{1}{q}} & 1 \le q < \infty \\
\sup_{t>0} t^{\frac{1}{p}} f^{**}(t) & q=\infty.
\end{array}
\right.
$$
(See Bennett-Sharpley \cite[Definition IV.4.4]{BS}).

Namely, from Lemma IV.4.5 in Bennett-Sharpley \cite{BS}  we have that
$$||f||_{L^{p,q}(\Omega)} \le ||f||_{L^{(p,q)}(\Omega)} \le \frac {p}{p-1} ||f||_{L^{p,q}(\Omega)}$$
for every $1\le q \le \infty.$

For a measurable vector-valued function $f=(f_1,\ldots, f_m): \Omega
\rightarrow  \mathbf{R}^m$ we say that $f \in L^{p,q}(\Omega;
\mathbf{R}^m)$ if and only if $f_i \in L^{p,q}(\Omega)$ for
$i=1,2,\ldots, m,$ if and only if $|f| \in L^{p,q}(\Omega)$ and we
define
\begin{equation*}
||f||_{L^{p,q}(\Omega;\mathbf{R}^m)}=||\,|f|\,||_{L^{p,q}(\Omega)}.
\end{equation*}
Similarly
\begin{equation*} ||f||_{L^{(p,q)}(\Omega;
\mathbf{R}^m)}=||\,|f|\,||_{L^{(p,q)}(\Omega)}.
\end{equation*}
 Obviously, it follows from the real-valued case that
$$||f||_{L^{p,q}(\Omega; \mathbf{R}^m)} \le ||f||_{L^{(p,q)}(\Omega; \mathbf{R}^m)}
\le \frac {p}{p-1} ||f||_{L^{p,q}(\Omega; \mathbf{R}^m)}$$ for every
$1 \le q \le \infty,$ and like in the real-valued case,
$||\cdot||_{L^{p,q}(\Omega; \mathbf{R}^m)}$ is already a norm when
$1\le q \le p,$ while it is a quasinorm when $p<q\le \infty.$

It is known that $(L^{p,q}(\Omega; \mathbf{R}^m),
||\cdot||_{L^{p,q}(\Omega; \mathbf{R}^m)})$ is a Banach space for
$1\le q \le p,$ while $(L^{p,q}(\Omega; \mathbf{R}^m),
||\cdot||_{L^{(p,q)}(\Omega; \mathbf{R}^m)})$ is a Banach space for
$1<p< \infty,$ $1\le q \le \infty.$ For more results on Lorentz spaces
we refer the reader to Bennett-Sharpley \cite[Chapter IV]{BS} and to
Stein-Weiss \cite[Chapter V]{SW}.

\subsection{Weak convergence of the $(p,q)$-norm and reflexivity of the Lorentz spaces}

\begin{Definition} \label{weak convergence in Lpq q finite} Let $\Omega$ be an open set in ${\mathbf{R}}^n,$
where $n \ge 1$ is an integer. Suppose $1<p<\infty$ and $1 \le q<\infty.$ We say that a sequence
$u_j$ in $L^{(p,q)}(\Omega)$ \textit{converges weakly} to a function $u \in L^{(p,q)}(\Omega)$ if
$$\int_{\Omega} v(x) u_j(x) \, dx \rightarrow \int_{\Omega} v(x) u(x) \, dx$$
whenever $v \in L^{(p', q')}(\Omega).$ There is an obvious interpretation in terms of the coordinate
functions for the weak convergence of vector-valued functions in $L^{(p,q)}(\Omega; {\mathbf{R}}^m),$ where $m \ge 1$ is an integer.
\end{Definition}

The spaces $L^{p,q}(\Omega;{\mathbf{R}}^m)$ are reflexive whenever $1<q<\infty$ and the dual of $L^{p,q}(\Omega; \mathbf{R}^m)$ is, up to equivalence of norms, the space $L^{p',q'}(\Omega; \mathbf{R}^m)$ for $1 \le q<\infty.$ See Bennett-Sharpley \cite[Theorem IV.4.7 and Corollary IV.4.8]{BS}, Hunt \cite[p.\ 259-262]{Hun} and the definition of the spaces $L^{p,q}(\Omega; \mathbf{R}^m).$
We notice that the terminology in the previous definition agrees with the usual weak convergence in the Banach space theory if $1 \le q<\infty.$

\subsection{Strict inclusions between Lorentz spaces}

\vspace{2mm}

\begin{Remark} \label{relation between Lpr and Lps}
It is known
(see Bennett-Sharpley \cite[Proposition IV.4.2]{BS}) that for every $p \in
(1,\infty)$ and $1\le r<s\le \infty$ there exists a constant $C(p,r,s)>0$ such
that
\begin{equation}\label{relation between the Lpr and the Lps norm}
||f||_{L^{p,s}(\Omega)} \le C(p,r,s) ||f||_{L^{p,r}(\Omega)}
\end{equation}
for all measurable functions $f \in L^{p,r}(\Omega).$ In particular,
$L^{p,r}(\Omega) \subset L^{p,s}(\Omega).$ Like in the real-valued
case, it follows that
\begin{equation}\label{relation between the Lpr and the Lps norm m ge 1}
||f||_{L^{p,s}(\Omega; \mathbf{R}^m)} \le C(p,r,s)
||f||_{L^{p,r}(\Omega; \mathbf{R}^m)}
\end{equation}
for every $m \ge 1$ integer and for all measurable functions $f \in
L^{p,r}(\Omega; \mathbf{R}^m),$ where $C(p,r,s)$ is the constant
from (\ref{relation between the Lpr and the Lps norm}). In particular,
$$L^{p,r}(\Omega; \mathbf{R}^m) \subset
L^{p,s}(\Omega;\mathbf{R}^m) \mbox{ for every $m \ge 1$ integer.}$$
\end{Remark}

The above inclusion is strict. See Ziemer \cite[p.\ 37, Exercise 1.7]{Zie} and
\cite[Theorems 3.4 and 3.5]{Cos4}.

\section{Sobolev-Lorentz Spaces}
\label{section Sobolev-Lorentz spaces}
This section is based in part on Chapter 3 of our book \cite{Cos3} and on Section 4 of our article \cite{Cos4}.

\subsection{The $H^{1, (p,q)}$ and $W^{1, (p,q)}$ Spaces}
In this subsection we recall the definition of the Sobolev-Lorentz spaces $H_{0}^{1, (p,q)}(\Omega),$
$H^{1, (p,q)}(\Omega),$ and $W^{1, (p,q)}(\Omega),$ where $\Omega \subset {\mathbf{R}}^n$ is an open set
and $n \ge 1$ is an integer. These spaces were studied extensively in Chapter 3 of our book \cite{Cos3}
(the case $n \ge 2$) and in Section 4 of our article \cite{Cos4} (the case $n \ge 1$).

For $1<p<\infty$ and $1\le q \le \infty$ we define the Sobolev-Lorentz space
$H^{1, (p,q)}(\Omega)$ as follows.
Let $r=\min(p,q).$ For a function $\phi \in
C^{\infty}(\Omega)$ we define its Sobolev-Lorentz $(p,q)$-norm by
$$||\phi||_{1, (p,q); \Omega}=\left(||\phi||_{L^{(p,q)}(\Omega)}^{r}
+||\nabla \phi||_{L^{(p,q)}(\Omega;
\mathbf{R}^n)}^{r}\right)^{1/r},$$ where $\nabla
\phi=(\partial_1 \phi, \ldots,
\partial_n \phi)$ is the gradient of $\phi.$
Similarly we define the Sobolev-Lorentz $p,q$-quasinorm of $\phi$ by
$$||\phi||_{1, p,q; \Omega}=\left(||\phi||_{L^{p,q}(\Omega)}^{r}
+||\nabla \phi||_{L^{p,q}(\Omega;
\mathbf{R}^n)}^{r}\right)^{1/r},$$
Then $H^{1, (p,q)}(\Omega)$ is defined as the completion of
$$\{\phi \in C^{\infty}(\Omega): ||\phi||_{1, (p,q); \Omega} < \infty \}$$
with respect to the norm $||\cdot||_{1, (p,q); \Omega}.$
Throughout the paper we use $||\cdot||_{H^{1, (p,q)}(\Omega)}$ instead
of $||\cdot||_{1, (p,q); \Omega}$ and $||\cdot||_{H^{1, p,q}(\Omega)}$ instead of
$||\cdot||_{1, p,q; \Omega}.$

The Sobolev-Lorentz space $H_{0}^{1, (p,q)}(\Omega)$ is defined as the closure of
$C_{0}^{\infty}(\Omega)$ in $H^{1, (p,q)}(\Omega)$.

From the discussion in subsection 4.1 of our paper \cite{Cos4}, we have that the Sobolev-Lorentz spaces $H_{0}^{1, (p,q)}(\Omega)$ and $H^{1,(p,q)}(\Omega)$  are reflexive Banach spaces when $1<q<\infty.$ From the same discussion it follows that $H_{0}^{1,(p,1)}(\Omega)$ and $H^{1,(p,1)}(\Omega)$ are non-reflexive Banach spaces.

Let $u \in L_{loc}^{1}(\Omega).$ For $i=1,\ldots,n$ a function $v
\in L_{loc}^{1}(\Omega)$ is called the \textit{$i$th weak partial
derivative of} $u$ and we denote $v= \partial_{i} u$ if
$$\int_{\Omega} \varphi(x) v(x) \, dx=-\int_{\Omega} \partial_{i}\varphi(x) u(x) \, dx$$
for all $\varphi \in C_{0}^{\infty}(\Omega).$

We define the Sobolev-Lorentz space $W^{1,(p,q)}(\Omega)$ by
\begin{equation*}
W^{1,(p,q)}(\Omega)=L^{(p,q)}(\Omega) \cap \{ u: \partial_{i} u \in
L^{(p,q)}(\Omega), \, i=1, \ldots, n \}.
\end{equation*}
The space $W^{1,(p,q)}(\Omega)$ is equipped with the norm
\begin{equation*}
||u||_{W^{1, (p,q)}(\Omega)}=||u||_{L^{(p,q)}(\Omega)}+
\sum_{i=1}^{n} ||\partial_{i} u||_{L^{(p,q)}(\Omega)},
\end{equation*}
which is clearly equivalent to
\begin{equation*}
 \left(||u||_{L^{(p,q)}(\Omega)}^r+ ||\nabla u||_{L^{(p,q)}(\Omega;
\mathbf{R}^n)}^r\right)^{1/r},
\end{equation*}
where $r=\min(p,q).$ Here $\nabla u$ is the distributional gradient of $u.$

In \cite[Theorem 4.8]{Cos4} we showed that $H^{1,(p,\infty)}(\Omega) \subsetneq W^{1,(p,\infty)}(\Omega)$ and that the spaces $H_{0}^{1,(p,\infty)}(\Omega),$ $H^{1,(p,\infty)}(\Omega)$ and
$W^{1,(p,\infty)}(\Omega)$ are not reflexive. Furthermore, in \cite[Theorem 4.11]{Cos4} we proved that $H^{1,(p,q)}(\Omega)=W^{1,(p,q)}(\Omega)$ whenever $1 \le q<\infty.$

The corresponding local space $H_{loc}^{1, (p,q)}(\Omega)$ is defined in the obvious manner:  $u$ is in $H_{loc}^{1,(p,q)}(\Omega)$ if and only if $u$ is in $H^{1, (p,q)}(\Omega')$ for every open set $\Omega' \subset \subset \Omega.$

Similarly, the local space $W_{loc}^{1, (p,q)}(\Omega)$ is defined as follows:  $u$ is in $W_{loc}^{1,(p,q)}(\Omega)$ if and only if $u$ is in $W^{1, (p,q)}(\Omega')$ for
every open set $\Omega' \subset \subset \Omega.$

For more details on these spaces including their basic properties we refer the readers to Chapter 3 of our book \cite{Cos3} and to Section 4 of our article \cite{Cos4}.

\subsection{Product rule}

Next we record the following lemma which says that the product between a function $u$ in $H^{1,(p,q)}(\Omega)$ and a function $\varphi$ in $C_{0}^{\infty}(\Omega)$ yields a function in
$H_{0}^{1,(p,q)}(\Omega)$ if $1<p<\infty$ and $1 \le q \le \infty.$
See also \cite[Lemma 4.9 and Theorem 4.11]{Cos4}).

\begin{Lemma} \label{Product Rule for W1pq}
Let $\Omega \subset {\mathbf{R}}^n$ be an open set, where $n \ge 1$ is an integer.
Suppose that $1<p<\infty$ and $1 \le q \le \infty.$
Suppose that $u \in H^{1,(p,q)}(\Omega)$ and that $\varphi \in
C_{0}^{\infty}(\Omega).$ Then $u \varphi \in H_{0}^{1,(p,q)}(\Omega)$
and $\nabla(u \varphi)=u \nabla \varphi+ \varphi \nabla u.$

\end{Lemma}

\begin{Remark} \label{boundedness product rule for H1pq and W1pq}
Lemma \ref{Product Rule for W1pq} easily implies
\begin{eqnarray*}
||u \varphi||_{W^{1,(p,q)}(\Omega)} &\le& ||\varphi||_{L^{\infty}(\Omega)} ||u||_{L^{(p,q)}(\Omega)} +||\varphi||_{L^{\infty}(\Omega)} \left(\sum_{i=1}^{n} ||\partial_i u||_{L^{(p,q)}(\Omega)}\right) \\
& &+ \left(\sum_{i=1}^n ||\partial_i \varphi||_{L^{\infty}(\Omega)}\right) ||u||_{L^{(p,q)}(\Omega)}\\
&\le& \left(||\varphi||_{L^{\infty}(\Omega)}+ \sum_{i=1}^{n} ||\partial_i \varphi||_{L^{\infty}(\Omega)}\right) \, ||u||_{W^{1,(p,q)}(\Omega)}
\end{eqnarray*}
for every $u \in H^{1,(p,q)}(\Omega)$
and
\begin{eqnarray*}
||u \varphi||_{H^{1,(p,q)}(\Omega)} &=& \left(||u \varphi||_{L^{(p,q)}(\Omega)}^r + ||\nabla (u \varphi)||_{L^{(p,q)}(\Omega;{\mathbf{R}}^n)}^r\right)^{1/r}\\
&=& \left(||u \varphi||_{L^{(p,q)}(\Omega)}^r + ||\varphi \nabla u + u \nabla \varphi)||_{L^{(p,q)}(\Omega;{\mathbf{R}}^n)}^r\right)^{1/r}\\
&\le& \left(||u \varphi||_{L^{(p,q)}(\Omega)}^r + ||\varphi \nabla u||_{L^{(p,q)}(\Omega;{\mathbf{R}}^n)}^r\right)^{1/r}+ ||u \nabla \varphi||_{L^{(p,q)}(\Omega; {\mathbf{R}}^n)}\\
&\le&||\varphi||_{L^{\infty}(\Omega)} ||u||_{H^{1,(p,q)}(\Omega)} + ||\nabla \varphi||_{L^{\infty}(\Omega)} ||u||_{L^{(p,q)}(\Omega; {\mathbf{R}}^n)} \\
&\le& \left(||\varphi||_{L^{\infty}(\Omega)}+ ||\nabla \varphi||_{L^{\infty}(\Omega)}\right) ||u||_{H^{1,(p,q)}(\Omega)}
\end{eqnarray*}
for every $u \in H^{1,(p,q)}(\Omega);$ here $1 \le q \le \infty$ and $r=\min(p,q),$ like in the definition of the $||\cdot||_{H^{1,(p,q)}(\Omega)}$ norm.
\end{Remark}

\subsection{Reflexivity results}

Next we recall the following reflexivity results from \cite{Cos3} concerning the Sobolev-Lorentz spaces, valid for all integers $n \ge 1$ and for all $q$ in $(1,\infty).$ Both these results are standard applications of Mazur's lemma.

\begin{Theorem} \label{HKM93 Thm130} {\rm(See \cite[Theorem V.20]{Cos0} and \cite[Theorem 3.5.2]{Cos3}).}
Let $1<p,q<\infty.$ Suppose that $\mathcal{K}$ is a convex and closed set of $H^{1,(p,q)}(\Omega).$ If $u_j \in \mathcal{K}$ is a sequence and if $u \in L^{(p,q)}(\Omega)$ and $v \in L^{(p,q)}(\Omega; {\mathbf{R}}^n)$ are functions such that $u_j \rightarrow u$ weakly in $L^{p,q}(\Omega)$ and $\nabla u_j \rightarrow \nabla u$ weakly in $L^{p,q}(\Omega; {\mathbf{R}}^n),$ then
$u \in \mathcal{K}$ and $v=\nabla u.$

\end{Theorem}

\begin{Theorem} \label{HKM93 Thm131} {\rm(See \cite[Theorem V.21]{Cos0} and \cite[Theorem 3.5.3]{Cos3}).}
Let $1<p,q<\infty.$ Suppose that $u_j$ is a bounded sequence in $H^{1,(p,q)}(\Omega).$ Then there is a subsequence $u_{j_i}$ and a function $u \in H^{1,(p,q)}(\Omega)$ such that $u_{j_i} \rightarrow u$ weakly in $L^{p,q}(\Omega)$ and
$\nabla u_{j_i} \rightarrow \nabla u$ weakly in $L^{p,q}(\Omega; {\mathbf{R}}^n).$ Moreover, if $u_j \in H_{0}^{1,(p,q)}(\Omega)$ for all $j \ge 1,$ then $u \in H_{0}^{1,(p,q)}(\Omega).$

\end{Theorem}

\section{Sobolev-Lorentz Capacity}
\label{section Sobolev-Lorentz capacity}

This section is based on Chapter 4 of our book \cite{Cos3}.
In \cite{Cos3} we studied the Sobolev-Lorentz relative and global capacities for
$1<p<\infty,$ $1 \le q \le \infty$ and $n>1$ integer. There we developed a capacity
theory based on the definition of Sobolev functions on $\mathbf{R}^n$ with respect to the Lorentz norm.
Basic properties of capacity, including monotonicity, countable subadditivity and several convergence
results were included there. All those results were proved in \cite{Cos3} for $n \ge 2$ but they can be extended to the case $n=1.$ We do it here, in this section of our paper.

\subsection{The Sobolev-Lorentz $(p,q)$ Relative Capacity}

Let $n \ge 1$ be an integer. Suppose $1<p<\infty$ and $1 \le q \le \infty.$ Let $\Omega \subset \mathbf{R}^n$ be a bounded open set
and let $E$ be a subset of $\Omega.$ The Sobolev-Lorentz $(p,q)$ relative capacity of the pair $(E, \Omega)$ is denoted
$${\mathrm{cap}}_{(p,q)}(E, \Omega)=\inf \, \{||\nabla u||_{L^{(p,q)}(\Omega; \mathbf{R}^n)}^{p}:
u \in {\mathcal{A}}(E, \Omega)\},$$ where
\begin{equation*}
{\mathcal{A}}(E, \Omega)=\{ u \in H_{0}^{1,(p,q)}(\Omega): u \ge 1\
\mbox{ in a neighborhood of } E\}.
\end{equation*}
We call ${\mathcal{A}}(E,\Omega)$ the \textit{set of admissible functions for the condenser} $(E, \Omega).$ If ${\mathcal{A}}(E,\Omega)=\emptyset$, we set $\mbox{cap}_{(p,q)}(E, \Omega)=\infty.$

Since $H_{0}^{1,(p,q)}(\Omega)$ is closed under truncations from below by $0$ and from above by $1$ and since these truncations do not increase the $(p,q)$-norm of the distributional gradients, it is enough to consider only those admissible functions $u$ for which $0 \le u \le 1$.

\subsubsection{Basic Properties of the $(p,q)$ Relative Capacity}

Usually, a capacity is a monotone and subadditive set function. The
following theorem will show, among other things, that this is true
in the case of the $(p,q)$ relative capacity. In our thesis \cite{Cos0}
we studied only the case $1<n=p<\infty.$ In \cite{Cos3} we extended
the results from \cite{Cos0} to the case $1<p<\infty$ and $n>1.$  The following
theorem generalizes Theorem V.23 from \cite{Cos0} and Theorem 4.1.1 from
\cite{Cos3} to the case $1<p<\infty$ and $n=1.$

\begin{Theorem}\label{Cap Thm (p,q) relative capacity}
{\rm(See \cite[Theorem V.23]{Cos0} and \cite[Theorem 4.1.1]{Cos3}).}
Let $n \ge 1$ be an integer. Suppose $1<p<\infty$ and $1 \le q \le \infty.$
Let $\Omega \subset \mathbf{R}^n$ be a bounded open set. The set
function $E \mapsto {\mathrm{cap}}_{(p,q)}(E, \Omega),$ $E \subset
\Omega,$ enjoys the following properties:

\par {\rm{(i)}}  If $E_{1} \subset E_{2},$ then ${\mathrm{cap}}_{(p,q)}(E_{1},
\Omega) \le {\mathrm{cap}}_{(p,q)}(E_{2}, \Omega).$

\par {\rm{(ii)}} If $\Omega_{1} \subset \Omega_{2}$ are open and bounded and $E \subset
\Omega_{1},$ then $${\mathrm{cap}}_{(p,q)}(E, \Omega_{2}) \le
{\mathrm{cap}}_{(p,q)}(E, \Omega_{1}).$$

\par {\rm{(iii)}} ${\mathrm{cap}}_{(p,q)}(E, \Omega)=\inf \, \{ {\mathrm{cap}}_{(p,q)}(U, \Omega):
E \subset U \subset \Omega, \,U \mbox{ open} \}.$

\par {\rm{(iv)}} If $K_{i}$ is a decreasing sequence of compact subsets of
$\Omega$ with $K=\bigcap_{i=1}^{\infty} K_{i},$ then
$${\mathrm{cap}}_{(p,q)}(K, \Omega)=\lim_{i \rightarrow \infty}
{\mathrm{cap}}_{(p,q)}(K_{i}, \Omega).$$

\par {\rm{(v)}} Suppose that $1<q<\infty.$ If $E_{1} \subset E_{2} \subset \ldots \subset E=\bigcup_{i=1}^{\infty} E_{i} \subset
\Omega,$ then
$${\mathrm{cap}}_{(p,q)}(E, \Omega)=\lim_{i \rightarrow \infty}
{\mathrm{cap}}_{(p,q)}(E_{i}, \Omega).$$

\par {\rm{(vi)}} If $E=\bigcup_{i=1}^{\infty} E_{i} \subset \Omega,$ then
$${\mathrm{cap}}_{(p,q)}(E, \Omega)^{1/p} \le
\sum_{i=1}^{\infty} {\mathrm{cap}}_{(p,q)}(E_{i}, \Omega)^{1/p}.$$

\end{Theorem}

\begin{proof} This result was proved in \cite{Cos3} for $n \ge 2.$ See \cite[Theorem 4.1.1]{Cos3}.
The proof of the case $n=1$ is very similar to the proof of \cite[Theorem 4.1.1]{Cos3} and omitted.

\end{proof}

The set function $\mathrm{cap}_{(p,q)}(\cdot, \Omega)$ satisfies properties (i), (iv), and (v) of Theorem \ref{Cap Thm (p,q) relative capacity}
whenever $1<p,q<\infty$ and $\Omega$ is a bounded open set in $\mathbf{R}^n,$ where $n \ge 1$ is an integer. Thus, $\mathrm{cap}_{(p,q)}(\cdot, \Omega)$
is a Choquet capacity (relative to $\Omega$) whenever $1<p,q<\infty$ and $\Omega$ is a bounded open set in $\mathbf{R}^n,$ where $n \ge 1$ is an integer. We may thus invoke an important capacitability theorem of Choquet and state the following result. See Doob \cite[Appendix II]{Doo}.

\begin{Theorem}\label{Choquet (p,q) relative capacity Thm 1<q<infty}
Let $\Omega$ be a bounded open set in $\mathbf{R}^n,$ where
$n \ge 1$ is an integer. Suppose $1<p,q<\infty.$
The set function $E \mapsto {\mathrm{cap}}_{(p,q)}(E, \Omega),$
$E \subset \Omega,$ is a Choquet capacity. In particular, all
Borel subsets (in fact, all analytic) subsets $E$ of
$\Omega$ are capacitable, i.e.
$${\mathrm{cap}}_{(p,q)}(E, \Omega)=\sup \, \{ {\mathrm{cap}}_{(p,q)}(K,
\Omega): K \subset E \mbox{ compact} \}.$$
\end{Theorem}

The set function $\mathrm{cap}_{(p,q)}(\cdot, \Omega)$ satisfies properties (i) and (iv) of Theorem \ref{Cap Thm (p,q) relative capacity}
whenever $q=1$ or $q=\infty.$ Like in Theorem \ref{Cap Thm (p,q) relative capacity}, the set $\Omega$ is bounded and open in $\mathbf{R}^n,$ where
$n \ge 1$ is an integer.

\begin{Question} \label{Question Choquet (p,q) relative capacity}
Let $\Omega$ be a bounded open set in $\mathbf{R}^n,$ where
$n \ge 1$ is an integer. Suppose $1<p<\infty.$
Is $\mathrm{cap}_{(p,q)}(\cdot, \Omega)$ a Choquet capacity when
$q=1$ or when $q=\infty$ ?
\end{Question}

We obtain a partial positive result later. Namely, we show later that if $\Omega \subset {\mathbf{R}}^n$
is a bounded open set, then $\mathrm{cap}_{(p,1)}(\cdot, \Omega)$ is a Choquet capacity
whenever $1 \le n<p<\infty$ or $1<n=p<\infty.$

\begin{Remark}\label{cap(K, Omega)=cap(bdry K, Omega) when 1 le q le infty and (p,q)}
Suppose $1 \le q \le \infty.$ The definition of the $(p,q)$-capacity easily implies
\begin{equation*}
{\mathrm{cap}}_{(p,q)}(K, \Omega)={\mathrm{cap}}_{(p,q)}(\partial K,
\Omega)
\end{equation*}
whenever $K$ is a compact set in $\Omega.$
\end{Remark}

\subsection{The Sobolev-Lorentz $p,q$ Relative Capacity}

Let $n \ge 1$ be an integer. Suppose $1<p<\infty$ and $1\le q \le \infty.$ We can introduce the ${p,q}$ relative capacity the way we
introduced the $(p,q)$ relative capacity. Let $\Omega \subset \mathbf{R}^n$ be a bounded and open set and let $E$ be a subset of $\Omega.$ We can define the Sobolev-Lorentz ${p,q}$ relative capacity of the pair $(E, \Omega)$ by
$${\mathrm{cap}}_{p,q}(E, \Omega)=
\inf \, \{||\nabla u||_{L^{p,q}(\Omega; \mathbf{R}^n)}^{p}: u \in
{\mathcal{A}}(E, \Omega)\}$$ where $${\mathcal{A}}(E, \Omega)=\{ u
\in H_{0}^{1,(p,q)}(\Omega): u \ge 1 \mbox{ in a neighborhood of }
E\}.$$ Like before, we call ${\mathcal{A}}(E, \Omega)$ the
\textit{set of admissible functions for the condenser}
$(E, \Omega).$ If ${\mathcal{A}}(E,\Omega)=\emptyset$, we set
$\mbox{cap}_{p,q}(E, \Omega)=\infty.$

Since $H_{0}^{1,(p,q)}(\Omega)$ is closed under truncations from below by $0$ and from
above by $1$ and since these truncations do not increase the $p,q$-quasinorm of the distributional gradients, it is enough to consider only those admissible functions $u$ for which $0 \le u \le 1$.

\subsubsection{Basic Properties of the $p,q$ Relative Capacity}

Usually, a capacity is a monotone and subadditive set function. The
following theorem will show, among other things, that this is true
in the case of the $p,q$ relative capacity. In our book \cite{Cos3} we studied the case
$1<p<\infty$ and $n \ge 2.$ We extend these results to the case $n=1.$ The following theorem
generalizes Theorem 3.2 from Costea-Maz'ya \cite{CosMaz} and Theorem 4.2.2 from \cite{Cos3}.

\begin{Theorem}\label{Cap Thm p,q relative capacity}
{\rm(See Costea-Maz'ya \cite[Theorem 3.2]{CosMaz} and \cite[Theorem 4.2.2]{Cos3}).}
Let $n \ge 1$ be an integer. Suppose $1<p<\infty$ and $1\le q \le \infty.$ Let
$\Omega \subset \mathbf{R}^n$ be a bounded open set.
The set function $E \mapsto {\mathrm{cap}}_{p,q}(E,
\Omega),$ $E \subset \Omega,$ enjoys the following properties:

\par {\rm{(i)}}  If $E_{1} \subset E_{2} \subset \Omega,$ then ${\mathrm{cap}}_{p,q}(E_{1},
\Omega) \le {\mathrm{cap}}_{p,q}(E_{2}, \Omega).$

\par {\rm{(ii)}} If $\Omega_{1} \subset \Omega_{2} \subset \mathbf{R}^n$ are open and $E \subset
\Omega_{1},$ then $${\mathrm{cap}}_{p,q}(E, \Omega_{2}) \le
{\mathrm{cap}}_{p,q}(E, \Omega_{1}).$$

\par {\rm{(iii)}} ${\mathrm{cap}}_{p,q}(E, \Omega)=\inf \, \{ {\mathrm{cap}}_{p,q}(U, \Omega):
E \subset U \subset \Omega, \,U \mbox{ open} \}.$

\par {\rm{(iv)}} If $K_{i}$ is a decreasing sequence of compact subsets of
$\Omega$ with $K=\bigcap_{i=1}^{\infty} K_{i},$ then
$${\mathrm{cap}}_{p,q}(K, \Omega)=\lim_{i \rightarrow \infty}
{\mathrm{cap}}_{p,q}(K_{i}, \Omega).$$

\par {\rm{(v)}} Suppose that $1<q \le p.$ If $E_{1} \subset E_{2} \subset \ldots \subset E=\bigcup_{i=1}^{\infty} E_{i} \subset
\Omega,$ then
$${\mathrm{cap}}_{p,q}(E, \Omega)=\lim_{i \rightarrow \infty}
{\mathrm{cap}}_{p,q}(E_{i}, \Omega).$$

\par {\rm{(vi)}} Suppose that $1 \le q \le p.$ If $E=\bigcup_{i=1}^{\infty} E_{i} \subset \Omega,$ then
$${\mathrm{cap}}_{p,q}(E, \Omega)^{q/p} \le \sum_{i=1}^{\infty} {\mathrm{cap}}_{p,q}(E_{i}, \Omega)^{q/p}.$$

\par {\rm{(vii)}} Suppose that $p < q < \infty.$ If $E=\bigcup_{i=1}^{\infty} E_{i} \subset \Omega,$ then $${\mathrm{cap}}_{p,q}(E, \Omega) \le \sum_{i=1}^{\infty} {\mathrm{cap}}_{p,q}(E_{i}, \Omega).$$

\par {\rm{(viii)}} Suppose that $q=\infty.$ Let $k \ge 1$ be an integer. If $E=\bigcup_{i=1}^{k} E_{i} \subset \Omega,$ then $${\mathrm{cap}}_{p,q}(E, \Omega) \le \sum_{i=1}^{k} {\mathrm{cap}}_{p,q}(E_{i}, \Omega).$$

\par {\rm{(ix)}} Suppose that $1 \le q \le \infty.$ If $\Omega_1$ and $\Omega_2$ are two disjoint open sets and $E \subset \Omega_1,$
then $${\mathrm{cap}}_{p,q}(E, \Omega_1 \cup \Omega_2)={\mathrm{cap}}_{p,q}(E, \Omega_1).$$

\par {{\rm(x)}} Suppose that $1 \le q \le p.$ Suppose that $\Omega_1,\ldots, \Omega_k$ are pairwise disjoint open sets and $E_i$ are subsets of $\Omega_i$ for $i=1,\ldots, k.$ If $E=\bigcup_{i=1}^{k} E_i$ and $\Omega=\bigcup_{i=1}^{k} \Omega_i,$ then
 $${\mathrm{cap}}_{p,q}(E, \Omega) \ge \sum_{i=1}^{k} {\mathrm{cap}}_{p,q}(E_i, \Omega_i).$$

\par {{\rm(xi)}} Suppose that $p < q < \infty.$  Suppose that $\Omega_i,\ldots, \Omega_k$ are pairwise disjoint open sets and $E_i$ are subsets of $\Omega_i$ for $i=1,\ldots, k.$ If $E=\bigcup_{i=1}^{k} E_i$ and $\Omega=\bigcup_{i=1}^{k} \Omega_i,$ then
 $${\mathrm{cap}}_{p,q}(E, \Omega)^{q/p} \ge \sum_{i=1}^{k} {\mathrm{cap}}_{p,q}(E_i, \Omega_i)^{q/p}.$$

\end{Theorem}

\begin{proof} This result was proved in \cite{Cos3} for $n \ge 2.$ See \cite[Theorem 4.2.2]{Cos3}.
The proof of the case $n=1$ is very similar to the proof of \cite[Theorem 4.2.2]{Cos3} and omitted.

\end{proof}

The set function $\mathrm{cap}_{p,q}(\cdot, \Omega)$ satisfies properties (i), (iv), and (v) of Theorem \ref{Cap Thm p,q relative capacity} whenever $1<q \le p<\infty$ and $\Omega$ is a bounded open set in $\mathbf{R}^n,$ where $n \ge 1$ is an integer. Thus, $\mathrm{cap}_{(p,q)}(\cdot, \Omega)$ is a Choquet capacity (relative to $\Omega$) whenever $1<q \le p<\infty$ and $\Omega$ is a bounded open set in $\mathbf{R}^n,$ where $n \ge 1$ is an integer. We may thus invoke an important capacitability theorem of Choquet and state the following result. See Doob \cite[Appendix II]{Doo}.

\begin{Theorem}\label{Choquet p,q relative capacity Thm 1<q le p}
Let $\Omega$ be a bounded open set in $\mathbf{R}^n,$ where
$n \ge 1$ is an integer. Suppose $1<q \le p< \infty.$ The set function $E \mapsto {\mathrm{cap}}_{p,q}(E, \Omega),$
$E \subset \Omega,$ is a Choquet capacity. In particular, all
Borel subsets (in fact, all analytic) subsets $E$ of
$\Omega$ are capacitable, i.e.
$${\mathrm{cap}}_{p,q}(E, \Omega)=\sup \, \{ {\mathrm{cap}}_{p,q}(K,
\Omega): K \subset E \mbox{ compact} \}.$$
\end{Theorem}

The set function $\mathrm{cap}_{p,q}(\cdot, \Omega)$ satisfies properties (i) and (iv) of
Theorem \ref{Cap Thm p,q relative capacity} whenever $q=1$ or $p<q \le \infty.$ Like in
Theorem \ref{Cap Thm p,q relative capacity}, the set $\Omega$ is bounded and open in $\mathbf{R}^n,$ where
$n \ge 1$ is an integer.

\begin{Question} \label{Question Choquet p,q relative capacity}
Let $1<p< \infty$ be fixed. Suppose that $\Omega$ is a bounded open set in $\mathbf{R}^n,$ where
$n \ge 1$ is an integer. Is $\mathrm{cap}_{p,q}(\cdot, \Omega)$ a Choquet capacity when
$q=1$ or when $p<q \le \infty$ ?
\end{Question}

We obtain a partial positive result later. Namely, we show later that if $\Omega \subset {\mathbf{R}}^n$
is a bounded open set, then $\mathrm{cap}_{p,1}(\cdot, \Omega)$ is a Choquet capacity
whenever $1 \le n<p<\infty$ or $1<n=p<\infty.$

\begin{Remark}\label{cap(K, Omega)=cap(bdry K, Omega) when 1 le q le infty p,q}
The definition of the $p,q$-capacity easily implies
\begin{equation*}
{\mathrm{cap}}_{p,q}(K, \Omega)={\mathrm{cap}}_{p,q}(\partial K, \Omega)
\end{equation*}
whenever $K$ is a compact set in $\Omega.$
\end{Remark}

\subsection{The Sobolev-Lorentz $(p,q)$ Global Capacity}
Let $n \ge 1$ be an integer. Suppose $1<p<\infty$ and $1\le q \le \infty.$
For a set $E \subset \mathbf{R}^n$ we define the Sobolev-Lorentz global $(p,q)$-capacity of $E$ by
$$\mathrm{Cap}_{(p,q)}(E)= \inf ||u||_{H^{1, (p,q)}(\mathbf{R}^n)}^p,$$ where $u$ runs through the set
$$S(E)=\{ u \in H_{0}^{1, (p,q)}(\mathbf{R}^n): u \ge 1 \mbox { in an open set containing } E\}.$$
If $S(E)=\emptyset,$ we set $\mathrm{Cap}_{(p,q)}(E)=\infty.$ It is obvious that the same number is obtained if the infimum in the definition is taken over $u \in S(E)$ with $0 \le u \le 1.$

\subsubsection{Basic Properties of the $(p,q)$ Global Capacity}

The following theorem summarizes the properties of the global Sobolev-Lorentz $(p,q)$-capacity, extending our results from \cite[Theorem 4.6.2]{Cos3} to the case $n=1.$

\begin{Theorem} \label{Cap Thm (p,q) global capacity}
{\rm(See \cite[Theorem 4.6.2]{Cos3}).}
Let $n \ge 1$ be an integer. Suppose $1<p<\infty$ and $1 \le q \le \infty.$
The set function $E \mapsto \mathrm{Cap}_{(p,q)}(E),$ $E \subset \mathbf{R}^n,$ has the following properties:

\par {\rm{(i)}}  If $E_{1} \subset E_{2},$ then ${\mathrm{Cap}}_{(p,q)}(E_{1}) \le {\mathrm{Cap}}_{(p,q)}(E_{2}).$

\par {\rm{(ii)}} ${\mathrm{Cap}}_{(p,q)}(E)=\inf \, \{ {\mathrm{Cap}}_{(p,q)}(U):
E \subset U \subset \mathbf{R}^n, \,U \mbox{ open} \}.$

\par {\rm{(iii)}} If $K_{i}$ is a decreasing sequence of compact subsets of
$\mathbf{R}^n$ with $K=\bigcap_{i=1}^{\infty} K_{i},$ then
$${\mathrm{Cap}}_{(p,q)}(K)=\lim_{i \rightarrow \infty}
{\mathrm{Cap}}_{(p,q)}(K_{i}).$$

\par {\rm{(iv)}} Suppose that $1<q<\infty.$ If $E_{1} \subset E_{2} \subset \ldots \subset E=\bigcup_{i=1}^{\infty} E_{i} \subset \mathbf{R}^n,$ then
$${\mathrm{Cap}}_{(p,q)}(E)=\lim_{i \rightarrow \infty}
{\mathrm{Cap}}_{(p,q)}(E_{i}).$$

\par {\rm{(v)}} If $E=\bigcup_{i=1}^{\infty} E_{i},$ then
$${\mathrm{Cap}}_{(p,q)}(E)^{1/p} \le
\sum_{i=1}^{\infty} {\mathrm{Cap}}_{(p,q)}(E_{i})^{1/p}.$$

\end{Theorem}

\begin{proof} This result was proved in \cite{Cos3} for $n \ge 2.$ See \cite[Theorem 4.6.2]{Cos3}.
The proof of the case $n=1$ is very similar to the proof of \cite[Theorem 4.6.2]{Cos3} and omitted.

\end{proof}

The set function $\mathrm{Cap}_{(p,q)}(\cdot)$ satisfies properties (i), (iii), and (iv) of Theorem \ref{Cap Thm (p,q) global capacity} whenever $n \ge 1$ is an integer and $1<p,q<\infty.$ Thus, this set function is a Choquet capacity whenever $n \ge 1$ is an integer and $1<p,q<\infty.$ We may thus invoke an important capacitability theorem of Choquet and state the following result. See Doob \cite[Appendix II]{Doo}.

\begin{Theorem}\label{Choquet Global (p,q) Cap Thm 1<q<infty} Let $n \ge 1$ be an integer. Suppose $1<p,q<\infty.$ The set function $E \mapsto {\mathrm{Cap}}_{(p,q)}(E),$ $E \subset \mathbf{R}^n,$ is a Choquet capacity. In particular, all Borel subsets (in fact, all analytic) subsets $E$ of $\mathbf{R}^n$ are capacitable, i.e.
$${\mathrm{Cap}}_{(p,q)}(E)=\sup \, \{ {\mathrm{Cap}}_{(p,q)}(K): K \subset E \mbox{ compact} \}.$$
\end{Theorem}

The set function $\mathrm{Cap}_{(p,q)}(\cdot)$ satisfies properties (i) and (iii) of Theorem \ref{Cap Thm (p,q) global capacity} whenever $q=1$ or $q=\infty.$ Like in Theorem \ref{Cap Thm (p,q) global capacity}, $n \ge 1$ is an integer.

\begin{Question} \label{Question Choquet (p,q) global capacity}
Let $n \ge 1$ be an integer. Suppose $1<p<\infty.$ Is $\mathrm{Cap}_{(p,q)}(\cdot)$ Choquet when
$q=1$ or when $q=\infty$ ?
\end{Question}

We obtain a partial positive result later. Namely, we show later that $\mathrm{Cap}_{(p,1)}(\cdot)$ is a Choquet capacity
whenever $1 \le n<p<\infty$ or $1<n=p<\infty.$

\subsection{The Sobolev-Lorentz $p,q$ Global Capacity}

Let $n \ge 1$ be an integer. Suppose $1<p<\infty$ and $1\le q \le \infty.$ We can introduce the global ${p,q}$-capacity the way we
introduced the global $(p,q)$-capacity. For a set $E \subset \mathbf{R}^n$ we define the Sobolev-Lorentz global $p,q$-capacity of $E$ by
$$\mathrm{Cap}_{p,q}(E)= \inf ||u||_{H^{1, p,q}(\mathbf{R}^n)}^p,$$ where $u$ runs through the set
$$S(E)=\{ u \in H_{0}^{1, (p,q)}(\mathbf{R}^n): u \ge 1 \mbox { in an open set containing } E\}.$$
If $S(E)=\emptyset,$ we set $\mathrm{Cap}_{p,q}(E)=\infty.$ It is obvious that the same number is obtained if the infimum in the definition is taken over $u \in S(E)$ with $0 \le u \le 1.$

\subsubsection{Basic Properties of the $p,q$ Global Capacity}

The following theorem summarizes the properties of the global Sobolev-Lorentz $p,q$-capacity, extending our results from \cite[Theorem 4.7.3]{Cos3} to the case $n=1.$

\begin{Theorem}\label{Cap Thm p,q global capacity}
{\rm(See \cite[Theorem 4.7.3]{Cos3}).}
Let $n \ge 1$ be an integer. Suppose $1<p<\infty$ and $1 \le q \le \infty.$ The set function $E \mapsto {\mathrm{Cap}}_{p,q}(E),$ $E \subset \mathbf{R}^n,$ enjoys the following properties:

\par {\rm{(i)}}  If $E_{1} \subset E_{2},$ then ${\mathrm{Cap}}_{p,q}(E_{1}) \le {\mathrm{Cap}}_{p,q}(E_{2}).$

\par {\rm{(ii)}} ${\mathrm{Cap}}_{p,q}(E)=\inf \, \{ {\mathrm{Cap}}_{p,q}(U):
E \subset U \subset \mathbf{R}^n, \,U \mbox{ open} \}.$

\par {\rm{(iii)}} If $K_{i}$ is a decreasing sequence of compact subsets of
$\mathbf{R}^n$ with $K=\bigcap_{i=1}^{\infty} K_{i},$ then
$${\mathrm{Cap}}_{p,q}(K)=\lim_{i \rightarrow \infty}
{\mathrm{Cap}}_{p,q}(K_{i}).$$

\par {\rm{(iv)}} Suppose that $1<q \le p.$ If $E_{1} \subset E_{2} \subset \ldots \subset E=\bigcup_{i=1}^{\infty} E_{i} \subset
\mathbf{R}^n,$ then
$${\mathrm{Cap}}_{p,q}(E)=\lim_{i \rightarrow \infty}
{\mathrm{Cap}}_{p,q}(E_{i}).$$

\par {\rm{(v)}} Suppose that $1 \le q \le p.$ If $E=\bigcup_{i=1}^{\infty} E_{i} \subset \mathbf{R}^n,$ then
$${\mathrm{Cap}}_{p,q}(E)^{q/p} \le \sum_{i=1}^{\infty} {\mathrm{Cap}}_{p,q}(E_{i})^{q/p}.$$

\par {\rm{(vi)}} Suppose that $p < q < \infty.$ If $E=\bigcup_{i=1}^{\infty} E_{i} \subset \mathbf{R}^n,$ then $${\mathrm{Cap}}_{p,q}(E) \le \sum_{i=1}^{\infty} {\mathrm{Cap}}_{p,q}(E_{i}).$$

\par {\rm{(vii)}} Suppose that $q=\infty.$ Let $k \ge 1$ be an integer. If $E=\bigcup_{i=1}^{k} E_{i} \subset \mathbf{R}^n,$ then $${\mathrm{Cap}}_{p,q}(E) \le \sum_{i=1}^{k} {\mathrm{Cap}}_{p,q}(E_{i}).$$

\end{Theorem}

\begin{proof} This result was proved in \cite{Cos3} for $n \ge 2.$ See \cite[Theorem 4.7.3]{Cos3}.
The proof of the case $n=1$ is very similar to the proof of \cite[Theorem 4.7.3]{Cos3} and omitted.

\end{proof}

The set function $\mathrm{Cap}_{p,q}(\cdot)$ satisfies properties (i), (iii), and (iv) of Theorem \ref{Cap Thm p,q global capacity} whenever $n \ge 1$ is an integer and $1<q \le p<\infty.$ Thus, this set function is a Choquet capacity whenever $n \ge 1$ is an integer and $1<q \le p<\infty.$ We may thus invoke an important capacitability theorem of Choquet and state the following result. See Doob \cite[Appendix II]{Doo}.

\begin{Theorem}\label{Choquet Global p,q Cap Thm 1<q le p} Let $n \ge 1$ be an integer. Suppose $1<q \le p<\infty.$
The set function $E \mapsto {\mathrm{Cap}}_{p,q}(E),$ $E \subset \mathbf{R}^n,$ is a Choquet capacity. In particular, all Borel subsets (in fact, all analytic) subsets $E$ of $\mathbf{R}^n$ are capacitable, i.e.
$${\mathrm{Cap}}_{p,q}(E)=\sup \, \{ {\mathrm{Cap}}_{p,q}(K): K \subset E \mbox{ compact} \}.$$
\end{Theorem}

The set function $\mathrm{Cap}_{p,q}(\cdot)$ satisfies properties (i) and (iii) of Theorem \ref{Cap Thm p,q global capacity} whenever $q=1$ or $p<q\le \infty.$ Like in Theorem \ref{Cap Thm p,q global capacity}, $n \ge 1$ is an integer.

\begin{Question} \label{Question Choquet p,q global capacity}
Let $n \ge 1$ be an integer. Suppose $1<p<\infty.$
Is $\mathrm{Cap}_{p,q}(\cdot)$ a Choquet capacity when
$q=1$ or when $p<q\le \infty$ ?
\end{Question}

We obtain a partial positive result later. Namely, we show later that
$\mathrm{Cap}_{p,1}(\cdot)$ is a Choquet capacity
when $1 \le n<p<\infty$ and when $1<n=p<\infty.$

\section{Sharp estimates for the Sobolev-Lorentz $n,1$ relative capacity}
\label{section Sharp estimates of the Sobolev-Lorentz n1 relative capacity}

In \cite{Cos1} we studied the $n,q$ relative capacity for $n \ge 2$ and $1 \le q \le \infty$ and we obtained sharp estimates for the $n,q$ relative capacity of the condensers $(\overline{B}(0,r), B(0,1))$ for small values of $r$ in $[0,1).$ See \cite[Theorem 3.11]{Cos1}. In this section we obtain a few new results. For instance, we obtain sharp estimates for the $n,1$ relative capacity of the aforementioned concentric condensers for ALL $r$ in $[0,1).$ See Theorem \ref{sharp estimates for the n1 relative capacity of a point} (i). Thus, we improve the estimates that we obtained in \cite{Cos1} for the $n,1$ relative capacity. In particular, we obtain the exact value for the $n,1$ capacity of a point relative to all its bounded open neighborhoods from ${\mathbf{R}}^n,$ a strictly positive number as we saw in \cite[Corollary 3.8]{Cos1}. See Theorem \ref{sharp estimates for the n1 relative capacity of a point} (ii).

Moreover, we obtain a new result concerning the $n,1$ global capacity. We show that this aforementioned value is also the value of the global $n,1$ capacity of any point from ${\mathbf{R}}^n.$ See Theorem \ref{sharp estimate for the n1 global capacity of a point}. This constant will also come into play later when we give a new proof of the embedding $H_{0}^{1,(n,1)}(\Omega) \hookrightarrow C(\overline{\Omega}) \cap L^{\infty}(\Omega),$ where $\Omega \subset {\mathbf{R}}^n$ is open and $n \ge 2$ is an integer. See Theorem \ref{continuous embedding of H01n1 into C cap Linfty}. This embedding is proved by using the exact value of the $n,1$ capacity of a point relative to any of its bounded open neighborhoods from ${\mathbf{R}}^n.$

In order to obtain the sharp estimates for the condensers $(\overline{B}(0,r), B(0,1))$ for all $r$ in $[0,1),$ we revisit Proposition 2.11 from \cite{Cos1} for $p=n>1$ and $q=1.$

\begin{Proposition} \label{Cos IUMJ Prop211}
Suppose $n \ge 2$ is an integer. Let $0 \le r < 1$ be fixed. Let $w: [\Omega_n r^n, \Omega_n] \rightarrow [0, \infty)$ be defined by $w(t)=\left(t/\Omega_n\right)^{1/n}.$ Suppose $f:[r,1] \rightarrow [0, \infty)$ is continuous and let $g:[\Omega_n r^n, \Omega_n] \rightarrow [0, \infty)$ be defined by $g(t)=f(w(t)).$ Then
$$||g||_{L^{n,1}([\Omega_n r^n, \Omega_n])} \ge n \Omega_n^{1/n} (1-r^n)^{-1/n'} \, ||f||_{L^{1}([r,1])}.$$

\end{Proposition}

\begin{proof} By applying \cite[Proposition 2.11]{Cos1} for $p=n>1$ and $q=1,$ we obtain
\begin{eqnarray*}
||g||_{L^{n,1}([\Omega_n r^n, \Omega_n])} &\ge& n \Omega_n ||(t/\Omega_n)^{-1/n'}||_{L^{n', \infty}([\Omega_n r^n, \Omega_n])}^{-1} \, ||f||_{L^1([r,1])}  \\
&=& n \Omega_n^{1/n} ||t^{-1/n'}||_{L^{n', \infty}([\Omega_n r^n, \Omega_n])}^{-1} \, ||f||_{L^1([r,1])}.
\end{eqnarray*}

We compute $||t^{-1/n'}||_{L^{n', \infty}([\Omega_n r^n, \Omega_n])}.$ We notice that
$$||t^{-1/n'}||_{L^{n', \infty}([\Omega_n r^n, \Omega_n])}=||(t+\Omega_n r^n)^{-1/n'}||_{L^{n', \infty}([0, \Omega_n (1-r^n)])}.$$

An easy computation shows that
\begin{eqnarray*}
||(t+\Omega_n r^n)^{-1/n'}||_{L^{n',\infty}([0, \Omega_n (1-r^n)])}&=&\sup_{0 \le t \le \Omega_n (1-r^n)} t^{1/n'} (t+\Omega_n r^n)^{-1/n'}\\
&=&\sup_{0 \le t \le \Omega_n (1-r^n)} \left(\frac{t}{t+\Omega_n r^n}\right)^{1/n'}=(1-r^n)^{1/n'}.
\end{eqnarray*}

This finishes the proof.

\end{proof}

Next we obtain sharp estimates for the $n,1$ relative capacity of the condensers $(\overline{B}(0,r), B(0,1))$ for ALL $r$ in $[0,1).$
Thus, we improve the estimates obtained in \cite[Theorem 3.11]{Cos1} for the $n,1$ relative capacity of the aforementioned concentric condensers.
As a consequence we obtain the exact value of the $n,1$ capacity of a point relative to all its bounded open neighborhoods from ${\mathbf{R}}^n.$

\begin{Theorem} \label{sharp estimates for the n1 relative capacity of a point}
Let $n \ge 2$ be an integer.

\par {\rm{(i)}} We have
$$n^n \Omega_n (1-r^n)^{1-n} \le {\mathrm{cap}}_{n,1}(\overline{B}(0,r), B(0,1)) \le n^n \Omega_n \frac{1-r^n}{(1-r)^n}$$
for every $0 \le r<1.$

\par {\rm{(ii)}} We have
$${\mathrm{cap}}_{n,1}(\{x\}, \Omega) = n^n \Omega_n$$
whenever $x \in {\mathbf{R}}^n$ and $\Omega$ is a bounded open set in ${\mathbf{R}}^n$ containing $x.$

\end{Theorem}

\begin{proof} We start by proving claim (i).

Let $r \in [0,1)$ be fixed. We want to compute the lower estimate. In order to do that, it is enough to consider via \cite[Lemma 3.6]{Cos1} only the admissible radial functions in $C_{0}^{\infty}(B(0,1))$ that are $1$ on a neighborhood of $\overline{B}(0,r).$ Let $u$ be such a function. There exists a function $f \in C^{\infty}([0,1])$ such that $u(x)=f(|x|)$ for every $x \in B(0,1).$ Hence $|\nabla u(x)|=|f'|(|x|)$ for every $x \in B(0,1).$ Moreover, $f'(t)=0$ for every $t \in [0,r].$ If we define $g:[0, \Omega_n] \rightarrow [0, \infty)$ by $g(t)=|f'|((t/\Omega_n)^{1/n}),$ we notice that $g$ is a continuous function compactly supported in $(\Omega_n r^n, \Omega_n).$ Moreover, since $|\nabla u(x)|=g(\Omega_n |x|^n)$ for every $x \in B(0,1),$ it follows that $|\nabla u|$ and $g$ have the same distribution function. From this and the fact that $g$ is supported in $(\Omega_n r^n, \Omega_n),$ we obtain
$$||\nabla u||_{L^{n,1}(B(0,1); {\mathbf{R}}^n)}=||g||_{L^{n,1}([\Omega_n r^n, \Omega_n])}.$$
But via Proposition \ref{Cos IUMJ Prop211} we have
$$||g||_{L^{n,1}([\Omega_n r^n, \Omega_n])} \ge n \Omega_n^{1/n} (1-r^n)^{-1/n'} ||f'||_{L^{1}([r,1])} $$
and since $||f'||_{L^{1}([r,1])} \ge f(r)-f(1)=1,$ we obtain
$$||\nabla u||_{L^{n,1}(B(0,1); {\mathbf{R}}^n)}=||g||_{L^{n,1}([\Omega_n r^n, \Omega_n])} \ge n \Omega_n^{1/n} (1-r^n)^{-1/n'}.$$

By taking the infimum over all admissible radial functions that are in $C_{0}^{\infty}(B(0,1)),$ we obtain the desired lower estimate.

Now we compute the upper estimate. Let $u_r: B(0,1) \rightarrow [0,1]$ be defined by
$$ u_{r}(x)=\left\{ \begin{array}{cl}
  1 & \mbox{if $0 \le |x| \le r$} \\
 \frac{1-|x|}{1-r}  & \mbox{if $r<|x|<1.$}
 \end{array}
 \right.$$

We notice that $u_{r}$ is a Lipschitz function on $B(0,1)$ that can be extended continuously by $0$ on
$\partial B(0,1).$ Thus, $u_{r}$ in $H_{0}^{1,(n,1)}(B(0,1)).$ Moreover, the function $\frac{1}{1-\varepsilon} u_{r}$ is admissible for the condenser $(\overline{B}(0,r), B(0,1))$ for every $\varepsilon \in (0,1).$

Thus, we have
$${\mathrm{cap}}_{n,1}(\overline{B}(0,r), B(0,1)) \le \frac{1}{(1-\varepsilon)^n} \, ||\nabla u_r||_{L^{n,1}(B(0,1);{\mathbf{R}}^n)}^n$$
for every $\varepsilon \in (0,1).$
By letting $\varepsilon \rightarrow 0,$ the above inequality yields
$${\mathrm{cap}}_{n,1}(\overline{B}(0,r), B(0,1)) \le ||\nabla u_r||_{L^{n,1}(B(0,1);{\mathbf{R}}^n)}^n.$$

An easy computation shows that
$$ |\nabla u_{r}(x)|=\left\{ \begin{array}{cl}
  0 & \mbox{if $0 \le |x| < r$} \\
 \frac{1}{1-r}  & \mbox{if $r<|x|<1.$}
 \end{array}
 \right.$$
Thus,
\begin{eqnarray*}
||\nabla u_r||_{L^{n,1}(B(0,1);{\mathbf{R}}^n)}&=&\frac{1}{1-r} || \chi_{B(0,1) \setminus \overline{B}(0,r)} ||_{L^{n,1}(B(0,1))}=\frac{1}{1-r} \int_{0}^{\Omega_n(1-r^n)} t^{\frac{1}{n}-1} dt\\
&=&\frac{1}{1-r} n [\Omega_n (1-r^n)]^{1/n}=n \Omega_n^{1/n} \, \frac{(1-r^n)^{1/n}}{1-r}.
\end{eqnarray*}
Hence, we obtain the desired upper estimate
$${\mathrm{cap}}_{n,1}(\overline{B}(0,r), B(0,1)) \le n^n \Omega_n \, \frac{(1-r^n)}{(1-r)^n}.$$
We obtained the desired lower and upper estimates for ${\mathrm{cap}}_{n,1}(\overline{B}(0,r), B(0,1))$. This finishes the proof of claim (i).

\vskip 2mm

We prove claim (ii) now. For $r=0,$ claim (i) yields
$${\mathrm{cap}}_{n,1}(\{0\}, B(0,1)) = n^n \Omega_n.$$
From this, the invariance of the $n,1$ relative under translations, \cite[Lemma 3.4]{Cos1} and \cite[Theorem 3.2 (ii)]{Cos1} (see also Theorem \ref{Cap Thm p,q relative capacity} (ii)), the desired conclusion follows. This finishes the proof.

\end{proof}

\subsection{Sharp estimates for the $n,1$ global capacity}
In Theorem \ref{sharp estimates for the n1 relative capacity of a point} we computed the exact value of the $n,1$ capacity of a point relative to all its bounded open neighborhoods from ${\mathbf{R}}^n$. We now prove that this strictly positive number is also the exact value of the $n,1$ global capacity of any point from ${\mathbf{R}}^n.$

\begin{Theorem} \label{sharp estimate for the n1 global capacity of a point}
Let $n \ge 2$ be an integer. We have
$${\rm{Cap}}_{n,1}(\{ x \})=n^n \Omega_n>0$$
for every $x \in {\mathbf{R}}^n.$

\end{Theorem}

\begin{proof} Since the $n,1$ global capacity is invariant under translation, it is enough to prove via Theorem \ref{sharp estimates for the n1 relative capacity of a point} (ii) that
$${\rm{Cap}}_{n,1}(\{0\})=n^n \Omega_n.$$

We prove first that $$n^n \Omega_n \le {\rm{Cap}}_{n,1}(\{0\}).$$

Fix $\varepsilon>0.$ Let $u \in S(\{0\})$ be an admissible function for $\{0\}$ with respect to the $n,1$ global capacity such that
$$||u||_{L^{n,1}({\mathbf{R}}^n)}+||\nabla u||_{L^{n,1}({\mathbf{R}}^n;{\mathbf{R}}^n)} <{\rm{Cap}}_{n,1}(\{0\})^{1/n}+\varepsilon.$$

Without loss of generality we can assume that $u  \in S(\{0\}) \cap C_{0}^{\infty}({\mathbf{R}}^n).$ Then $u$ is compactly supported
in a bounded open set $U \subset {\mathbf{R}}^n$ that contains the origin.
It is easy to see that $S(\{0\}) \cap C_{0}^{\infty}(U)={\mathcal{A}}(\{0\}, U) \cap C_{0}^{\infty}(U).$
Thus, $u \in {\mathcal{A}}(\{0\}, U).$

Therefore we have via Theorem \ref{sharp estimates for the n1 relative capacity of a point} (ii)
\begin{eqnarray*}
n \Omega_n^{1/n}={\rm{cap}}_{n,1}( \{0\}, U)^{1/n} &\le& ||\nabla u||_{L^{n,1}(U; {\mathbf{R}}^n)}=||\nabla u||_{L^{n,1}({\mathbf{R}}^n; {\mathbf{R}}^n)}\\
 &\le& ||u||_{L^{n,1}({\mathbf{R}}^n)}+||\nabla u||_{L^{n,1}({\mathbf{R}}^n;{\mathbf{R}}^n)}\\
 &<& {\rm{Cap}}_{n,1}(\{0\})^{1/n}+\varepsilon.
\end{eqnarray*}
By letting $\varepsilon \rightarrow 0,$ we see that indeed $n^n \Omega_n \le {\rm{Cap}}_{n,1}(\{0\}).$

Conversely, we want to show that ${\rm{Cap}}_{n,1}(\{0\}) \le n^n \Omega_n.$

Let $u: {\mathbf{R}}^n \rightarrow [0,1]$ be defined by
$$ u(x)=\left\{ \begin{array}{cl}
  1 - |x| & \mbox{if $0 \le |x| \le 1$} \\
  0  & \mbox{if $|x|>1.$}
 \end{array}
 \right.$$

We notice that $u$ is a Lipschitz function on ${\mathbf{R}}^n$ that is supported in $\overline{B}(0,1).$ Thus, $u \in H_{0}^{1,(n,1)}({\mathbf{R}}^n).$
Moreover, the function $\frac{1}{1-\varepsilon} u$ is in $S(\{0\})$ for every $\varepsilon$ in $(0,1).$
For every $r$ in $(0,\infty)$ we define $u_r: {\mathbf{R}}^n \rightarrow [0,1]$ by $u_r(x)=u(\frac{x}{r}).$
We notice that $u_r$ is a Lipschitz function on ${\mathbf{R}}^n$ that is supported in $\overline{B}(0,r)$ for every $r>0.$
Thus, $u_r \in H_{0}^{1,(n,1)}({\mathbf{R}}^n)$ for every $r>0.$ Moreover, the function $\frac{1}{1-\varepsilon} u_r$ is in $S(\{0\})$
for every $r>0$ and for every $\varepsilon$ in $(0,1).$

It is easy to see that
$$||u_r||_{L^{n,1}({\mathbf{R}}^n)}=r ||u||_{L^{n,1}({\mathbf{R}}^n)} \mbox{ and } ||\nabla u_r||_{L^{n,1}({\mathbf{R}}^n; {\mathbf{R}}^n)}=||\nabla u||_{L^{n,1}({\mathbf{R}}^n; {\mathbf{R}}^n)}$$
for every $r>0.$
Thus,
\begin{eqnarray*}
{\rm{Cap}}_{n,1}(\{0\})^{1/n} &\le& \frac{1}{1-\varepsilon} \left(||u_r||_{L^{n,1}({\mathbf{R}}^n)}+||\nabla u_r||_{L^{n,1}({\mathbf{R}}^n;{\mathbf{R}}^n)}\right)\\
&=& \frac{1}{1-\varepsilon} \left(r||u||_{L^{n,1}({\mathbf{R}}^n)}+||\nabla u||_{L^{n,1}({\mathbf{R}}^n;{\mathbf{R}}^n)}\right)
\end{eqnarray*}
for every $r>0$ and for every $\varepsilon$ in $(0,1).$ By letting $r \rightarrow 0$ and $\varepsilon \rightarrow 0,$ we obtain
$${\rm{Cap}}_{n,1}(\{0\}) \le ||\nabla u||_{L^{n,1}({\mathbf{R}}^n; {\mathbf{R}}^n)}^n.$$

An easy computation shows that
$$ |\nabla u(x)|=\left\{ \begin{array}{cl}
  1  & \mbox{if $0<|x|<1$}\\
  0 & \mbox{if $|x|>1.$}
 \end{array}
 \right.$$
Thus,
\begin{equation*}
||\nabla u||_{L^{n,1}({\mathbf{R}}^n;{\mathbf{R}}^n)}= || \chi_{B(0,1)} ||_{L^{n,1}({\mathbf{R}}^n)}=\int_{0}^{\Omega_n} t^{\frac{1}{n}-1} dt
=n \Omega_n^{1/n}.
\end{equation*}
Therefore, we obtain
$${\rm{Cap}}_{n,1}(\{0\}) \le ||\nabla u||_{L^{n,1}({\mathbf{R}}^n;{\mathbf{R}}^n)}^n=n^n \Omega_n.$$
This finishes the proof of the theorem.

\end{proof}

The following theorem gives a new proof of the embedding $H_{0}^{1,(n,1)}(\Omega) \hookrightarrow C(\overline{\Omega}) \cap L^{\infty}(\Omega)$ whenever $\Omega \subset {\mathbf{R}}^n$ is an open set and $n \ge 2.$ This embedding was proved before by Stein in \cite{Ste} and by Cianchi-Pick in \cite{CiPi1}.
See \cite[Theorem 3.5 (i)]{CiPi1}. Our approach is different. Our proof uses the theory of the $n,1$ relative capacity in ${\mathbf{R}}^n$, $n \ge 2$ and the the exact value of the $n,1$ relative capacity of a point relative to all its bounded and open neighborhoods from ${\mathbf{R}}^n.$ This value has been obtained in Theorem \ref{sharp estimates for the n1 relative capacity of a point} (ii).

\begin{Theorem} \label{continuous embedding of H01n1 into C cap Linfty}
Let $\Omega \subset {\mathbf{R}}^n$ be an open set, where $n \ge 2$ is an integer.
If $u \in H_{0}^{1,(n,1)}(\Omega),$ then $u$ has a version $u^{*} \in C(\overline{\Omega}) \cap L^{\infty}(\Omega)$ and
\begin{equation} \label{Linfty norm of u le Ln1 norm of nabla u}
||u||_{L^{\infty}(\Omega)} \le \frac{1}{n \Omega_n^{1/n}} ||\nabla u||_{L^{n,1}(\Omega; {\mathbf{R}}^n)}.
\end{equation}
Moreover, if $\Omega \neq {\mathbf{R}}^n,$ then $u^{*}=0$ on $\partial \Omega.$

\end{Theorem}

\begin{proof} First we prove (\ref{Linfty norm of u le Ln1 norm of nabla u}) for the functions in $C_{0}^{\infty}(\Omega).$
So let $u$ be a function in $C_{0}^{\infty}(\Omega).$ We can assume without loss of generality that $u$ is not identically zero.
Let $U \subset \subset \Omega$ be a bounded open set such that $\mbox{ supp } u \subset \subset U.$

Fix $\lambda \in (0,1).$ Let $O_{\lambda}=\{ x \in \Omega: |u(x)|> \lambda ||u||_{L^{\infty}(\Omega)} \}.$
Then $O_{\lambda} \subset \subset U$ is a bounded nonempty open set and the Lipschitz function $\frac{|u|}{ \lambda ||u||_{L^{\infty}(\Omega)}}$
is supported in $U$ and is admissible for the condenser $(O_\lambda, U)$ with respect to the $n,1$-capacity.

Let $x \in O_{\lambda}.$ The monotonicity of the ${\rm{cap}}_{n,1}(\cdot, U)$ set function and Theorem \ref{sharp estimates for the n1 relative capacity of a point} (ii) imply

$$n \Omega_n^{1/n}={\rm{cap}}_{n,1}(\{x\}, U)^{1/n} \le {\rm{cap}}_{n,1}(O_{\lambda}, U)^{1/n} \le  \frac{||\nabla u||_{L^{n,1}(\Omega; {\mathbf{R}}^n)}}{\lambda ||u||_{L^{\infty}(\Omega)}}.$$
Thus,
$$||u||_{L^{\infty}(\Omega)} \le \frac{1}{n \Omega_n^{1/n}} \cdot \lambda^{-1} ||\nabla u||_{L^{n,1}(\Omega; {\mathbf{R}}^n)}$$
for every $\lambda \in (0,1).$ By letting $\lambda \rightarrow 1,$ we obtain the desired inequality for $u \in C_{0}^{\infty}(\Omega).$

Suppose now that $u \in H_{0}^{1,(n,1)}(\Omega).$ Let $(u_k)_{k \ge 1} \subset C_{0}^{\infty}(\Omega)$ be a sequence that converges to $u$ in $H_{0}^{1,(n,1)}(\Omega).$

We can assume without loss of generality that the sequence $(u_k)_{k \ge 1} \subset C_{0}^{\infty}(\Omega)$ is chosen such that $u_k$ converges pointwise to $u$ almost everywhere in $\Omega,$ $\nabla u_k$ converges pointwise to $\nabla u$ almost everywhere in $\Omega,$ and such that
$$||u_{k+1}-u_k||_{L^{n,1}(\Omega)}+||\nabla u_{k+1} - \nabla u_k||_{L^{n,1}(\Omega; {\mathbf{R}}^n)} < 2^{-2k}, \, \forall k \ge 1.$$

By applying the inequality (\ref{Linfty norm of u le Ln1 norm of nabla u}) to the smooth functions $u_k$ and $u_{k+1}-u_{k}$ that are compactly supported in $\Omega$ for every $k \ge 1$, we obtain
\begin{eqnarray}
\label{Linfty norm of uk le Ln1 norm of nabla uk}
|u_k(x)| &\le& \frac{1}{n \Omega_n^{1/n}} ||\nabla u_k||_{L^{n,1}(\Omega; {\mathbf{R}}^n)} \mbox{ and }\\
\label{Linfty norm of uk+1-uk le Ln1 norm of nabla uk+1-uk}
|u_{k+1}(x)-u_k(x)| &\le& \frac{1}{n \Omega_n^{1/n}} ||\nabla u_{k+1}-\nabla u_k||_{L^{n,1}(\Omega; {\mathbf{R}}^n)}
\end{eqnarray}
for all $x$ in $\Omega$ and for every $k \ge 1.$

From the choice of the sequence $u_k$ and (\ref{Linfty norm of uk+1-uk le Ln1 norm of nabla uk+1-uk}), it follows that the sequence $u_k$ is uniformly fundamental on $\Omega.$ Thus, $u_k$ converges uniformly in $\Omega$ to a function $v \in C(\Omega).$ Since the functions $u_k$ are in $C_{0}^{\infty}(\Omega),$
we can assume without loss of generality that they are in $C(\overline{\Omega}).$ This is trivial when $\Omega={\mathbf{R}}^n;$ when $\Omega \neq {\mathbf{R}}^n$ (that is, when $\partial \Omega \neq \emptyset$), we set all the functions $u_k$ to be $0$ on $\partial \Omega.$
Thus, the sequence $u_k$ is uniformly fundamental in $\overline{\Omega}$ and its uniform limit $v$ is continuous on $\overline{\Omega}.$ Moreover,
if $\Omega \neq {\mathbf{R}}^n,$ then $v$ is $0$ on $\partial \Omega.$

Since $u_k$ converges pointwise almost everywhere to $u$ in $\Omega$ and uniformly to $v$ in $\Omega,$ it follows that $u=v$ almost everywhere in $\Omega.$ Consequently, $v \in H_{0}^{1,(n,1)}(\Omega)$ and $\nabla u=\nabla v$ almost everywhere in $\Omega.$

By letting $k \rightarrow \infty$ in (\ref{Linfty norm of uk le Ln1 norm of nabla uk}), we obtain
\begin{eqnarray*}
|v(x)|=\lim_{k \rightarrow \infty} |u_k(x)| &\le& \frac{1}{n \Omega_n^{1/n}} \lim_{k \rightarrow \infty} ||\nabla u_k||_{L^{n,1}(\Omega; {\mathbf{R}}^n)}\\
&=&\frac{1}{n \Omega_n^{1/n}} ||\nabla u||_{L^{n,1}(\Omega; {\mathbf{R}}^n)}\\
&=&\frac{1}{n \Omega_n^{1/n}} ||\nabla v||_{L^{n,1}(\Omega; {\mathbf{R}}^n)}
\end{eqnarray*}
for every $x$ in $\Omega.$ Since $u=v$ almost everywhere in $\Omega,$ this implies that $u \in L^{\infty}(\Omega)$ and (\ref{Linfty norm of u le Ln1 norm of nabla u}) holds.
This finishes the proof of the theorem.

\end{proof}

\begin{Remark} \label{Linfty norm of u equal Ln1 norm of nabla u} Suppose $n \ge 2$ is an integer. We can see easily that
we have equality in (\ref{Linfty norm of u le Ln1 norm of nabla u}) for the functions $u_r$ that were used in the proof of Theorem \ref{sharp estimate for the n1 global capacity of a point}. In the aforementioned theorem, for a fixed $r$ in $(0,\infty)$ the function $u_r: {\mathbf{R}}^n \rightarrow [0,1]$ was defined by
$$ u_r(x)=\left\{ \begin{array}{cl}
  1 - \frac{|x|}{r} & \mbox{if $0 \le |x| \le r$} \\
  0  & \mbox{if $|x|>r.$}
 \end{array}
 \right.$$

\end{Remark}

As a consequence of Theorem \ref{continuous embedding of H01n1 into C cap Linfty}, we now show that every function in $H^{1,(n,1)}_{loc}(\Omega)$ has a version that is continuous on $\Omega.$ It is pretty clear to see that this result also follows as a consequence of the aforementioned results obtained by Stein in \cite{Ste} and by Cianchi-Pick in \cite{CiPi1}.

\begin{Proposition} \label{continuous version for functions in H1n1}
Let $\Omega \subset {\mathbf{R}}^n$ be an open set, where $n \ge 2$ is an integer. Suppose that $u \in H_{loc}^{1, (n,1)}(\Omega).$
Then $u$ has a version $u^{*} \in C(\Omega).$

\end{Proposition}

\begin{proof} Choose open sets $\emptyset=\Omega_0 \subsetneq \Omega_j \subset \subset \Omega_{j+1}, j \ge 1$ such that $\bigcup_j \Omega_j=\Omega.$ Like in the proof of \cite[Theorem 4.11]{Cos4} (see also Heinonen-Kilpelainen-Martio \cite[Lemma 1.15]{HKM}), we construct a sequence
$\psi_j, j \ge 1$ such that $\psi_j \in C_{0}^{\infty}(\Omega_{j+1} \setminus \overline{\Omega}_{j-1})$ for every $j \ge 1$ and $\sum_j \psi_j \equiv 1$ on $\Omega.$

We notice via Lemma \ref{Product Rule for W1pq} that $u \psi_j \in H_{0}^{1, (n,1)}(\Omega)$ is compactly supported in $\Omega$ for all $j \ge 1.$ By applying Theorem \ref{continuous embedding of H01n1 into C cap Linfty} to the sequence $(u \psi_j)_{j \ge 1},$ we find a continuous version $(u \psi_j)^{*}$ of $u \psi_j$ that is compactly supported in $\Omega$ for every $j \ge 1.$ Then
$u^{*}:=\sum_j (u \psi_j)^{*}$ is a version of $u=\sum_j u \psi_j.$ Since on every bounded open set $U \subset \subset \Omega$ only finitely many of the functions $(u \psi_j)^{*}$ are non-vanishing, it follows immediately that $u^{*}$ is in fact continuous on $\Omega.$ This finishes the proof.

\end{proof}

\section{Bounded sequences in non-reflexive Sobolev-Lorentz spaces}
\label{section Bounded sequences in non-reflexive Sobolev-Lorentz spaces}

Whenever we proved a Monotone Convergence Theorem for the relative and global Sobolev-Lorentz $(p,q)$ capacities associated to reflexive spaces $H_{0}^{1,(p,q)}(\Omega),$ $1<q<\infty,$ we always used the fact (via Theorems \ref{HKM93 Thm130}-\ref{HKM93 Thm131}) that every bounded sequence $(u_k)_{k \ge 1} \subset H_{0}^{1,(p,q)}(\Omega)$ has a subsequence $(u_{k_i})$ that converges weakly in $H^{1,(p,q)}(\Omega)$ to a function $u \in H_{0}^{1,(p,q)}(\Omega).$

We know that the spaces $H^{1,(p,1)}(\Omega)$ and $H_{0}^{1,(p,1)}(\Omega)$ are not reflexive. See for instance the discussion from Section 4.1 in our paper \cite{Cos4}. In this section we prove a weak convergence theorem concerning $H^{1,(p,1)}(\Omega)$ and $H_{0}^{1,(p,q)}(\Omega)$ that is similar to Theorem \ref{HKM93 Thm131}. The weak convergence result concerning $H^{1,(p,1)}(\Omega)$ is valid
whenever $1<p<\infty.$ See Theorem \ref{Bdd in H01p1 weak limit in H01p1} (i). The weak convergence result concerning $H_{0}^{1,(p,1)}(\Omega)$ is valid whenever $1 \le n<p<\infty$ or $1<n=p<\infty.$ See Theorem \ref{Bdd in H01p1 weak limit in H01p1} (ii).

\begin{Theorem} \label{Bdd in H01p1 weak limit in H01p1}
Let $\Omega \subset {\mathbf{R}}^n$ be an open set, where $n \ge 1$ is an integer. Suppose that $1<p,q<\infty.$ Let $u$ be a function in $H^{1,(p,q)}(\Omega)$ and let $(u_k)_{k \ge 1} \subset H^{1, (p,1)}(\Omega)$ be a sequence that is bounded in $H^{1, (p,1)}(\Omega).$ Suppose that $u_k$ converges to $u$ weakly in $L^{(p,q)}(\Omega)$ and that $\nabla u_k$ converges to $\nabla u$ weakly in $L^{(p,q)}(\Omega; {\mathbf{R}}^n).$

\par {\rm(i)} We have that $u$ is in $H^{1,(p,1)}(\Omega).$  Moreover, the sequence $u_k$ converges to $u$ weakly in $L^{(p,s)}(\Omega),$ while the sequence $\nabla u_k$ converges to $\nabla u$ weakly in $L^{(p,s)}(\Omega; {\mathbf{R}}^n)$ for every $1<s<\infty.$ Also, we have

\begin{equation} \label{weak convergence in Lp1 for u and partial i u}
\int_{\Omega} u_k(x) \varphi(x) \, dx \rightarrow  \int_{\Omega} u(x) \varphi(x) \, dx \mbox{ and }
\int_{\Omega} \partial_i u_k(x) \varphi(x) \, dx \rightarrow  \int_{\Omega} \partial_i u(x) \varphi(x) \, dx, i=1, \ldots, n
\end{equation}

for every simple function $\varphi \in L^{p',\infty}(\Omega).$ Furthermore, we have

\begin{eqnarray} \label{H1p1 norm of u le liminf H1p1 norm of uk}
||u||_{H^{1,(p,1)}(\Omega)} &\le& \liminf_{k \rightarrow \infty} ||u_k||_{H^{1,(p,1)}(\Omega)} \\
\label{H1p1 quasinorm of u le liminf H1p1 quasinorm of uk}
||u||_{H^{1,p,1}(\Omega)} &\le& \liminf_{k \rightarrow \infty} ||u_k||_{H^{1,p,1}(\Omega)} \\
\label{Lp1 norm of nabla u le liminf Lp1 norm of nabla uk}
||\nabla u||_{L^{(p,1)}(\Omega; {\mathbf{R}}^n)} &\le& \liminf_{k \rightarrow \infty} ||\nabla u_k||_{L^{(p,1)}(\Omega;{\mathbf{R}}^n)} \mbox { and } \\
\label{Lp1 quasinorm of nabla u le liminf Lp1 quasinorm of nabla uk}
||\nabla u||_{L^{p,1}(\Omega; {\mathbf{R}}^n)} &\le& \liminf_{k \rightarrow \infty} ||\nabla u_k||_{L^{p,1}(\Omega; {\mathbf{R}}^n)}.
\end{eqnarray}

\par {\rm(ii)} Suppose that $1 \le n<p<\infty$ or $1<n=p<\infty.$ If $u_j \in H_{0}^{1,(p,1)}(\Omega)$ for all $j \ge 1,$ then $u \in H_{0}^{1,(p,1)}(\Omega).$

\end{Theorem}

\begin{proof} (i) We recall that $H^{1,(p,r)}(\Omega) \subset H^{1,(p,s)}(\Omega)$ and $H_{0}^{1,(p,r)}(\Omega) \subset H_{0}^{1,(p,s)}(\Omega)$ whenever $1 \le r <s \le \infty.$ (See \cite[Theorem 4.3]{Cos4}).
Moreover, from Remark \ref{relation between Lpr and Lps}, (\ref{relation between the Lpr and the Lps norm m ge 1}) and the definition of the Sobolev-Lorentz spaces and norms on $\Omega,$ it follows that there exists a constant $C(p,r,s)>0$ such that
\begin{equation}
\label{relation between the Lpr and the Lps Sobolev norm} ||v||_{H^{1,(p,s)}(\Omega)} \le C(p,r,s) ||v||_{H^{1,(p,r)}(\Omega)}
\end{equation}
for every $v \in H^{1,(p,r)}(\Omega).$
Thus, any sequence that is bounded in $H^{1,(p,1)}(\Omega)$ is also bounded in $H^{1,(p,s)}(\Omega)$
whenever $1<s<\infty.$ Similarly, any sequence that is bounded in $H_{0}^{1,(p,1)}(\Omega)$ is also bounded in $H_{0}^{1,(p,s)}(\Omega)$ whenever $1<s<\infty.$ The spaces $H^{1,(p,s)}(\Omega)$ and $H_{0}^{1,(p,s)}(\Omega)$ are reflexive whenever $1<p,s<\infty.$ See the discussion before Theorem 4.1 from our paper \cite{Cos4}.

Let $q$ in $(1,\infty)$ be fixed. From Theorem \ref{HKM93 Thm131} it follows that for any sequence $\widetilde{u}_k$ that is bounded in $H^{1,(p,q)}(\Omega)$ there is a subsequence $\widetilde{u}_{k_i}$ and a function $\widetilde{u} \in H^{1,(p,q)}(\Omega)$ such that $\widetilde{u}_{k_i}$ converges weakly to $\widetilde{u}$ in $L^{(p,q)}(\Omega)$ and $\nabla \widetilde{u}_{k_i}$ converges weakly to $\nabla \widetilde{u}$ in $L^{(p,q)}(\Omega; {\mathbf{R}}^n).$ Moreover, from Theorem \ref{HKM93 Thm131} it follows that $\widetilde{u}$ is in $H_{0}^{1,(p,q)}(\Omega)$ if all the functions $\widetilde{u}_j$ are in $H_{0}^{1, (p,q)}(\Omega).$

If such a sequence $\widetilde{u}_k$ is bounded in $H^{1,(p,1)}(\Omega),$ from (\ref{relation between the Lpr and the Lps Sobolev norm}), Theorem \ref{HKM93 Thm131} and from the discussion in the previous paragraph we obtain the existence of a subsequence $\widetilde{u}_{k_i}$ and of a function $\widetilde{u} \in H^{1,(p,q)}(\Omega)$ such that $\widetilde{u}_{k_i}$ converges weakly in $L^{p,q}(\Omega)$ to $\widetilde{u}$ and such that $\nabla \widetilde{u}_{k_i}$ converges weakly to $\nabla \widetilde{u}$ in $L^{p,q}(\Omega; {\mathbf{R}}^n).$ Moreover, from the discussion in the previous paragraph and from (\ref{relation between the Lpr and the Lps Sobolev norm}) it follows that $\widetilde{u}$ is in $H_{0}^{1,(p,q)}(\Omega)$ if all the functions $\widetilde{u}_j$ are in $H_{0}^{1, (p,1)}(\Omega).$

However, since the subsequence $\widetilde{u}_{k_i}$ is bounded in the non-reflexive space $H^{1,(p,1)}(\Omega),$ it is also bounded via (\ref{relation between the Lpr and the Lps Sobolev norm}) in the reflexive spaces $H^{1,(p,s)}(\Omega)$ for every $1<s<\infty.$ From this, the fact that $\widetilde{u}$ is the weak limit of the subsequence $\widetilde{u}_{k_i}$ in $H^{1,(p,q)}(\Omega),$ from the definition of the spaces $L^{p,s}(\Omega; {\mathbf{R}}^m),$ from Hunt \cite[p.\ 258]{Hun} and from Bennett-Sharpley \cite[Theorem IV.4.7 and Corollary IV.4.8]{BS}, it follows in fact that $\widetilde{u} \in H^{1,(p,s)}(\Omega)$ for every $1<s<\infty.$ Moreover, $\widetilde{u}_{k_i}$ converges weakly in $L^{p,s}(\Omega)$ to $\widetilde{u}$ and $\nabla \widetilde{u}_{k_i}$ converges weakly to $\nabla \widetilde{u}$ in $L^{p,s}(\Omega; {\mathbf{R}}^n)$ for every $1<s<\infty.$

Thus, if we have a bounded sequence $u_k$ in $H^{1,(p,1)}(\Omega)$ that converges weakly in $H^{1,(p,q)}(\Omega)$ to a function $u \in H^{1,(p,q)}(\Omega),$ the above argument shows that
$u$ belongs to $H^{1,(p,s)}(\Omega),$  $u_{k}$ converges weakly in $L^{p,s}(\Omega)$ to $u$ and $\nabla u_{k}$ converges weakly to $\nabla u$ in $L^{p,s}(\Omega; {\mathbf{R}}^n)$ for every $1<s<\infty.$
Moreover, (\ref{weak convergence in Lp1 for u and partial i u}) holds for $u$ and for the sequence $u_k.$
If in addition the sequence $(u_k)_{k \ge 1} \subset H_{0}^{1,(p,1)}(\Omega)$ is bounded in $H_{0}^{1,(p,1)}(\Omega),$ from the previous argument and Theorem \ref{HKM93 Thm131} it follows that the function $u$ is in fact in $H_{0}^{1,(p,s)}(\Omega)$ for all $1<s< \infty.$

From (\ref{weak convergence in Lp1 for u and partial i u}) it follows easily via Fatou's Lemma and via the H\"{o}lder inequality for Lorentz spaces (see \cite[Theorem 2.3]{Cos1} and/or \cite[Theorem 3.7]{Cos4}) that
\begin{eqnarray}
\label{u is a functional on Lpinfty abs cont star}
\left|\int_{\Omega} u(x) \varphi(x) \, dx\right| &\le& \left(\liminf_{k \rightarrow \infty} ||u_k||_{L^{p,1}(\Omega)}\right) \, ||\varphi||_{L^{p',\infty}(\Omega)} \mbox{ and } \\
\label{partial i u is a functional on Lpinfty abs cont star}
\left|\int_{\Omega} \partial_i u(x) \varphi(x) \, dx\right| &\le& \left(\liminf_{k \rightarrow \infty} ||\partial_i u_k||_{L^{p,1}(\Omega)}\right) \, ||\varphi||_{L^{p',\infty}(\Omega)}, i=1, \ldots, n
\end{eqnarray}
for every simple function $\varphi \in L^{p',\infty}(\Omega).$

From (\ref{u is a functional on Lpinfty abs cont star}) and Bennett-Sharpley \cite[Proposition I.3.13, Theorems I.4.1 and IV.4.7]{BS} it follows that $u$ is in $L^{p,1}(\Omega).$ From (\ref{partial i u is a functional on Lpinfty abs cont star}) and Bennett-Sharpley \cite[Proposition I.3.13, Theorems I.4.1 and IV.4.7]{BS} it follows that $\partial_i u$ is in $L^{p,1}(\Omega)$ for $i=\overline{1,n}.$ Thus, $u$ is in $W^{1,(p,1)}(\Omega).$ Since $W^{1,(p,1)}(\Omega)=H^{1, (p,1)}(\Omega)$ (see \cite[Theorem 4.11]{Cos4}), it follows that $u$ is indeed in $H^{1,(p,1)}(\Omega).$

Thus, we finally showed that $u$ is in $H^{1,(p,1)}(\Omega).$

Now we prove that
\begin{eqnarray}
\label{Lp1 norm of u le liminf Lp1 norm of uk}
||u||_{L^{(p,1)}(\Omega)} &\le& \liminf_{k \rightarrow \infty} ||u_k||_{L^{(p,1)}(\Omega)} \mbox{ and } \\
\label{Lp1 norm of partial i u le liminf Lp1 norm of partial i uk}
||\partial_i u||_{L^{(p,1)}(\Omega)} &\le& \liminf_{k \rightarrow \infty} ||\partial_i u_k||_{L^{(p,1)}(\Omega)}, i=1, \ldots, n.
\end{eqnarray}

From (\ref{weak convergence in Lp1 for u and partial i u}) and Stein-Weiss \cite[Lemma V.3.17 (i) and (iii)]{SW} it follows that whenever $0<t \le |\Omega|$ there exist Lebesgue measurable sets $E_t$ and $E_{t,i} \subset \Omega, i=1, \ldots n$ such that $|E_t|=|E_{t,i}|=t, \, i=1, \ldots, n$ and such that

\begin{eqnarray}
\label{u** le liminf uk**}
u^{**}(t)&=&\frac{1}{t} \int_{0}^{t} u^{*}(s) \, ds=\frac{1}{t} \int_{E_t} |u(x)| \, dx \le \liminf_{k \rightarrow \infty} \frac{1}{t} \int_{E_t} |u_k(x)| \, dx \\
         &\le& \liminf_{k \rightarrow \infty} \frac{1}{t} \int_{0}^{t} |u_k|^{*}(s) \, ds = \liminf_{k \rightarrow \infty} u_k^{**}(t) \mbox{ and } \nonumber \\
\label{partial i u** le liminf partial i uk**}
|\partial_i u|^{**}(t)&=&\frac{1}{t} \int_{0}^{t} |\partial_i u|^{*}(s) \, ds=\frac{1}{t} \int_{E_{t,i}} |\partial_i u(x)| \, dx \le \liminf_{k \rightarrow \infty} \frac{1}{t} \int_{E_{t,i}} |\partial_i u_k(x)| \, dx \\
         &\le& \liminf_{k \rightarrow \infty} \frac{1}{t} \int_{0}^{t} |\partial_i u_k|^{*}(s) \, ds =
         \liminf_{k \rightarrow \infty} |\partial_i u_k|^{**}(t), \, i=1, \ldots, n. \nonumber
\end{eqnarray}

Now (\ref{Lp1 norm of u le liminf Lp1 norm of uk}) follows from (\ref{u** le liminf uk**}) and Fatou's Lemma, while (\ref{Lp1 norm of partial i u le liminf Lp1 norm of partial i uk}) follows from (\ref{partial i u** le liminf partial i uk**}) and Fatou's Lemma.

Since we do not know whether the sequence $(u_k, \nabla u_k)$ converges weakly to $(u, \nabla u)$
in $L^{p,1}(\Omega) \times L^{p,1}(\Omega;{\mathbf{R}}^n)$ or not, we cannot use the weak-$\star$ lower semicontinuity of the $L^{p,1}$ quasinorm in order to derive (\ref{H1p1 norm of u le liminf H1p1 norm of uk})-(\ref{Lp1 quasinorm of nabla u le liminf Lp1 quasinorm of nabla uk}). Thus, we have to use a different approach in order to obtain (\ref{H1p1 norm of u le liminf H1p1 norm of uk})-(\ref{Lp1 quasinorm of nabla u le liminf Lp1 quasinorm of nabla uk}).

We can choose subsequences $u_{k,1}, u_{k,2}, u_{k,3}$ and $u_{k,4}$ such that that
\begin{eqnarray}
\label{lim H1p1 norm of uk1 equal liminf H1p1 norm of uk}
\lim_{k \rightarrow \infty} ||u_{k,1}||_{H^{1,(p,1)}(\Omega)}&=&\liminf_{k \rightarrow \infty} ||u_k||_{H^{1,(p,1)}(\Omega)}, \\
\label{lim H1p1 quasinorm of uk2 equal liminf H1p1 quasinorm of uk}
\lim_{k \rightarrow \infty} ||u_{k,2}||_{H^{1,p,1}(\Omega)}&=&\liminf_{k \rightarrow \infty} ||u_k||_{H^{1,p,1}(\Omega)}, \\
\label{lim Lp1 norm of nabla uk3 equal liminf Lp1 norm of nabla uk}
\lim_{k \rightarrow \infty} ||\nabla u_{k,3}||_{L^{(p,1)}(\Omega; {\mathbf{R}}^n)}&=&\liminf_{k \rightarrow \infty} ||\nabla u_k||_{L^{(p,1)}(\Omega; {\mathbf{R}}^n)} \mbox{ and } \\
\label{lim Lp1 quasinorm of nabla uk4 equal liminf Lp1 quasinorm of nabla uk}
\lim_{k \rightarrow \infty} ||\nabla u_{k,4}||_{L^{p,1}(\Omega; {\mathbf{R}}^n)}&=&\liminf_{k \rightarrow \infty} ||\nabla u_k||_{L^{p,1}(\Omega; {\mathbf{R}}^n)}.
\end{eqnarray}

We can apply the Mazur lemma to the sequences $(u_{k,i}, \nabla u_{k,i}),$ $i=1, \ldots, 4$ with respect to the reflexive space $L^{(p,q)}(\Omega) \times L^{(p,q)}(\Omega; \mathbf{R}^n)$ to obtain sequences $v_{k,i}$ of convex combinations of $u_{k,i},$ $i=1, \ldots, 4$ such that $v_{k,i} \rightarrow u$ in $H^{1,(p,q)}(\Omega),$ $v_{k,i} \rightarrow u$ almost everywhere in $\Omega$ and $\nabla v_{k,i} \rightarrow \nabla u$ almost everywhere in $\Omega,$ $i=\overline{1,4}.$

We present here the construction argument for the Mazur lemma.
Let $k_{0} \ge 1$ and $1 \le i \le 4$ be fixed. Since every subsequence of $(u_{k,i}, \nabla u_{k,i})$
converges to $(u, \nabla u)$ weakly in $L^{(p,q)}(\Omega) \times L^{(p,q)}(\Omega; \mathbf{R}^n),$ we may
use the Mazur lemma for the subsequence $u_{k,i}, k \ge k_{0}$ with respect to
$L^{(p,q)}(\Omega) \times L^{(p,q)}(\Omega; \mathbf{R}^n).$ We
obtain a finite convex combination $v_{k_0,i}$ of the functions
$u_{k,i}, k\ge k_{0},$
$$v_{k_0,i}=\sum_{j=k_0}^{j_{k_0,i}} \lambda_{k_0,j,i} u_{j,i}, \quad \lambda_{k_0,j,i} \ge 0 \mbox { and } \sum_{j=k_0}^{j_{k_0,i}} \lambda_{k_0,j,i}=1$$
as close to $u$ as we want in $H^{1,(p,q)}(\Omega)$
(but not necessarily in $H^{1,(p,1)}(\Omega)$).

A finite convex combination of functions from $H^{1, (p,1)}(\Omega)$ is a function from
$H^{1, (p,1)}(\Omega).$ Moreover, for every $i=1, \ldots, 4$ we have $(v_{k,i})_{k \ge 1} \subset H_{0}^{1,(p,1)}(\Omega)$ if $(u_k)_{k \ge 1} \subset H_{0}^{1,(p,1)}(\Omega).$
Passing to subsequences if necessary, we may assume for every $i=1, \ldots, 4$ that $v_{k,i} \rightarrow u$ almost everywhere in $\Omega,$
$\nabla v_{k,i} \rightarrow \nabla u$ almost everywhere in $\Omega$ and that
\begin{equation*}
||v_{k+1,i}-v_{k,i}||_{L^{(p,q)}(\Omega)} + ||\nabla v_{k+1,i}-\nabla v_{k,i}||_{L^{(p,q)}(\Omega; \mathbf{R}^n)} < 2^{-2k}
\end{equation*}
for every $k \ge 1.$

This ends the construction of the sequences $(v_{k,i})_{k \ge 1} \subset H^{1,(p,1)}(\Omega)$, $i=1, \ldots, 4$ by using the Mazur lemma with respect to $L^{(p,q)}(\Omega) \times L^{(p,q)}(\Omega; \mathbf{R}^n).$

Now we finish the prove of claim (i).

\vspace{2mm}

From the convexity of both the $p,1$-norm and the $(p,1)$-norm, the choice of the sequences $u_{k,i},$ and the definition of the functions $v_{k,i},$
$i=1, \ldots, 4$ we obtain
\begin{eqnarray}
\label{liminf H1p1 norm of vk1 le liminf H1p1 norm of uk}
\liminf_{k \rightarrow \infty} ||v_{k,1}||_{H^{1,(p,1)}(\Omega)} &\le& \liminf_{k \rightarrow \infty} ||u_k||_{H^{1,(p,1)}(\Omega)}, \\
\label{liminf H1p1 quasinorm of vk2 le liminf H1p1 quasinorm of uk}
\liminf_{k \rightarrow \infty} ||v_{k,2}||_{H^{1,p,1}(\Omega)} &\le& \liminf_{k \rightarrow \infty} ||u_k||_{H^{1,p,1}(\Omega)}, \\
\label{liminf Lp1 norm of nabla vk3 le liminf Lp1 norm of nabla uk}
\liminf_{k \rightarrow \infty} ||\nabla v_{k,3}||_{L^{(p,1)}(\Omega; {\mathbf{R}}^n)} &\le& \liminf_{k \rightarrow \infty} ||\nabla u_k||_{L^{(p,1)}(\Omega; {\mathbf{R}}^n)} \mbox{ and }\\
\label{liminf Lp1 quasinorm of nabla vk4 le liminf Lp1 quasinorm of nabla uk}
\liminf_{k \rightarrow \infty} ||\nabla v_{k,4}||_{L^{p,1}(\Omega; {\mathbf{R}}^n)} &\le& \liminf_{k \rightarrow \infty} ||\nabla u_k||_{L^{p,1}(\Omega; {\mathbf{R}}^n)}.
\end{eqnarray}

Since the sequences $v_{k,i}, k \ge 1$ converge pointwise almost everywhere on $\Omega$ to $u$ for every $i=1, \ldots, 4,$ it follows via Bennett-Sharpley \cite[Proposition II.1.7]{BS} and via Fatou's Lemma that for every $i=1, \ldots, 4$ we have
\begin{equation}
\label{u* le liminf vki* and u** le liminf vki**}
u^{*}(t) \le \liminf_{k \rightarrow \infty} v_{k,i}^{*}(t) \mbox{ and } u^{**}(t) \le \liminf_{k \rightarrow \infty} v_{k,i}^{**}(t) \mbox{ for every } t>0.
\end{equation}

Similarly, since the sequences $\nabla v_{k,i}, k \ge 1$ converge pointwise almost everywhere on $\Omega$ to $\nabla u$ for every $i=1, \ldots, 4,$ it follows via Bennett-Sharpley \cite[Proposition II.1.7]{BS} and via  Fatou's Lemma that for every $i=1, \ldots, 4$ we have
\begin{equation}
\label{nabla u* le liminf nabla vki* and nabla u** le liminf nabla vki**}
|\nabla u|^{*}(t) \le \liminf_{k \rightarrow \infty} |\nabla v_{k,i}|^{*}(t) \mbox{ and } |\nabla u|^{**}(t) \le \liminf_{k \rightarrow \infty} |\nabla v_{k,i}|^{**}(t) \mbox{ for every } t>0.
\end{equation}

Moreover, from (\ref{liminf H1p1 norm of vk1 le liminf H1p1 norm of uk}), (\ref{liminf H1p1 quasinorm of vk2 le liminf H1p1 quasinorm of uk}), (\ref{liminf Lp1 norm of nabla vk3 le liminf Lp1 norm of nabla uk}) and (\ref{liminf Lp1 quasinorm of nabla vk4 le liminf Lp1 quasinorm of nabla uk}) we have via (\ref{u* le liminf vki* and u** le liminf vki**}) and (\ref{nabla u* le liminf nabla vki* and nabla u** le liminf nabla vki**}) and via Fatou's Lemma
\begin{eqnarray*}
||u||_{H^{1,(p,1)}(\Omega)} &\le& \liminf_{k \rightarrow \infty} ||v_{k,1}||_{H^{1,(p,1)}(\Omega)} \le \liminf_{k \rightarrow \infty} ||u_k||_{H^{1,(p,1)}(\Omega)} \\
||u||_{H^{1,p,1}(\Omega)} &\le& \liminf_{k \rightarrow \infty} ||v_{k,2}||_{H^{1,p,1}(\Omega)} \le \liminf_{k \rightarrow \infty} ||u_k||_{H^{1,p,1}(\Omega)} \\
||\nabla u||_{L^{(p,1)}(\Omega; {\mathbf{R}}^n)} &\le& \liminf_{k \rightarrow \infty} ||\nabla v_{k,3}||_{L^{(p,1)}(\Omega; {\mathbf{R}}^n)} \le
\liminf_{k \rightarrow \infty} ||\nabla u_k||_{L^{(p,1)}(\Omega; {\mathbf{R}}^n)} \mbox{ and } \\
||\nabla u||_{L^{p,1}(\Omega; {\mathbf{R}}^n)} &\le& \liminf_{k \rightarrow \infty} ||\nabla v_{k,4}||_{L^{p,1}(\Omega; {\mathbf{R}}^n)} \le \liminf_{k \rightarrow \infty} ||\nabla u_k||_{L^{p,1}(\Omega; {\mathbf{R}}^n)}.
\end{eqnarray*}

This finishes the proof of claim (i).

\vspace{2mm}

(ii) Now we prove the second claim of the theorem. Before we start the proof of claim (ii), we recall that in part (i) we proved that if we have a sequence $(u_k)_{k \ge 1} \subset H_{0}^{1,(p,1)}(\Omega)$ that is bounded in $H^{1,(p,1)}(\Omega)$ such that $u_k$ converges to $u$ weakly in $L^{(p,q)}(\Omega)$ and such that $\nabla u_k$ converges to $\nabla u$ weakly in $L^{(p,q)}(\Omega; {\mathbf{R}}^n)$ for some $q$ in $(1, \infty),$ then $u$ in $H_{0}^{1,(p,s)}(\Omega)$ whenever $1<s<\infty.$ Moreover, we also proved in part (i) that the sequence $u_k$ converges to $u$ weakly in $L^{(p,s)}(\Omega)$ and that the sequence $\nabla u_k$ converges to $\nabla u$ weakly in $L^{(p,s)}(\Omega; {\mathbf{R}}^n)$ whenever $1<s<\infty.$ Furthermore, we also proved that $u$ is in $H^{1,(p,1)}(\Omega).$ The result in part (i) is valid whenever
$1<p<\infty.$

Now we show that if the sequence $(u_k)_{k \ge 1} \subset H_{0}^{1,(p,1)}(\Omega)$ is bounded in $H^{1,(p,1)}(\Omega),$  $u_k$ converges to $u$ weakly in $L^{(p,q)}(\Omega)$ and $\nabla u_k$ converges to $\nabla u$ weakly in $L^{(p,q)}(\Omega; {\mathbf{R}}^n)$ for some $q$ in $(1, \infty),$ then $u \in H_{0}^{1, (p,1)}(\Omega)$ provided that $1 \le n<p<\infty$ or $1<n=p<\infty.$

Under the hypotheses of claim (ii) we can assume without loss of generality via the discussion at the beginning of the proof of this theorem together with our previous results \cite[Theorem 5.5 (iii)]{Cos4} when $1=n<p<\infty$ and respectively \cite[Theorem 5.6 (iv)]{Cos4} when $1<n<p<\infty$ that $u$ and all
the functions $u_k$ and $v_{k,i},$ $k \ge 1, i=1, \ldots, 4$ are H\"{o}lder continuous with exponent $1-\frac{n}{p}$ on the closed set $\overline{\Omega}.$ Moreover, if $\partial \Omega \neq \emptyset,$ all these functions are $0$ on $\partial \Omega.$ Furthermore, from the construction of the sequences $v_{k,i},$ from our previous results \cite[Theorem 5.5 (iii)]{Cos4} when $1=n<p<\infty$ and respectively \cite[Theorem 5.6 (iv)]{Cos4} when $1<n<p<\infty,$ it follows that the sequences $v_{k,i}$ converge uniformly to $u$ on compact subsets of $\Omega$ (or respectively uniformly on $\overline{\Omega}$ if $\Omega$ is bounded).

\smallskip

When $1<n=p<\infty,$ we can assume without loss of generality via Theorem \ref{continuous version for functions in H1n1} that $u$ is continuous on $\Omega.$ Moreover, when $1<n=p<\infty$ we can also assume without loss of generality via Theorem \ref{continuous embedding of H01n1 into C cap Linfty} that all the functions $u_k$ and $v_{k,i}$ are continuous on $\overline{\Omega}$ and in addition, if $\partial \Omega \neq \emptyset,$ all of them are $0$ on $\partial \Omega.$

We notice that claim (ii) holds trivially for all $p$ in $(1,\infty)$ when $\Omega={\mathbf{R}}^n$ via our previous result \cite[Theorem 4.12]{Cos4} or when $u$ is compactly supported in $\Omega$ via our previous result \cite[Lemma 4.21]{Cos4}.

\smallskip

We need to prove claim (ii) when $\Omega \neq {\mathbf{R}}^n.$ We have to consider two separate cases, $\Omega$ bounded and $\Omega \neq {\mathbf{R}}^n$ unbounded.

Before we differentiate between the cases $\Omega$ bounded and $\Omega \neq {\mathbf{R}}^n$ unbounded, we extend the functions $u_k$, $v_{k,i}$ and $u$ by zero on the nonempty set ${\mathbf{R}}^n \setminus \Omega$ and we denote these extensions by $\widetilde{u}_{k},$ $\widetilde{v}_{k,i}$ and $\widetilde{u}$ respectively. From the discussion at the beginning of the proof of claim (ii) it follows that all the functions $\widetilde{u}_{k}$ and $\widetilde{v}_{k,i}$ are continuous on ${\mathbf{R}}^n.$ From our previous result \cite[Proposition 5.2]{Cos4} and from the hypothesis of claim (ii), we see that the sequence $\widetilde{u}_k$ is bounded in $H_{0}^{1,(p,1)}({\mathbf{R}}^n),$ $\widetilde{u}_k$ converges weakly to $\widetilde{u}$ in $L^{(p,q)}({\mathbf{R}}^n)$ and $\nabla \widetilde{u}_{k}$ converges weakly to $\nabla \widetilde{u}$ in $L^{(p,q)}({\mathbf{R}}^n; {\mathbf{R}}^n).$ Thus, via claim (i) it follows that $\widetilde{u} \in H_{0}^{1,(p,1)}({\mathbf{R}}^n).$ This implies via \cite[Theorem 5.5 (iii)]{Cos4} when $1=n<p<\infty,$ via \cite[Theorem 5.6 (iv)]{Cos4} when $1<n<p<\infty$ and respectively via Theorem \ref{continuous version for functions in H1n1} when $1<n=p<\infty$ that $\widetilde{u}$ has a version $\widetilde{u}^{*} \in C({\mathbf{R}}^n).$ Since $\widetilde{u}^{*}$ is continuous on ${\mathbf{R}}^n,$
$u$ is continuous on $\Omega$ and the restriction of $\widetilde{u}^{*}$ to $\Omega$ is a version of $u,$
it follows immediately that $u=\widetilde{u}^{*}=\widetilde{u}$ everywhere in $\Omega.$

Now we consider the cases $\Omega$ bounded and $\Omega \neq {\mathbf{R}}^n$ unbounded separately.

\smallskip

Case 1. We start with the case when $\Omega$ is bounded.

We saw already that $u=\widetilde{u}^{*}=\widetilde{u}$ everywhere in $\Omega.$ Since $\widetilde{u}$ is zero everywhere on ${\mathbf{R}}^n \setminus \Omega \supset \partial \Omega$ and its version $\widetilde{u}^{*}$ is in $C_{0}({{\mathbf{R}}^n}),$ it follows in fact that $\widetilde{u}^{*}$ is zero everywhere in ${\mathbf{R}}^n \setminus \Omega \supset \partial \Omega$ along with $\widetilde{u}.$ Thus, $\widetilde{u}=\widetilde{u}^{*}$ everywhere on ${\mathbf{R}}^n,$ both of them are zero on $\partial \Omega$ (along with $u$) and $\widetilde{u}^{*}=\widetilde{u}=u$ in $\Omega.$ Thus, we proved that $u$ is a continuous function on $\Omega$ that extends continuously by $0$ on $\partial \Omega.$ This implies via our previous result \cite[Lemma 4.21]{Cos4} that $u \in H_{0}^{1,(p,1)}(\Omega).$ Thus, claim (ii) holds when $\Omega$ is bounded provided that $1 \le n<p<\infty$ or $1<n=p<\infty.$

\smallskip

Case 2. We consider now the case when $\Omega \neq {\mathbf{R}}^n$ is unbounded.

Without loss of generality we can assume that $0 \in \Omega \neq {\mathbf{R}}^n.$ Like in the proof of \cite[Theorem 4.12]{Cos4}, we choose a sequence of 2-Lipschitz smooth functions $(\phi_j)_{j \ge 1} \subset C_{0}^{\infty}({\mathbf{R}}^n)$ such that $0 \le \phi_j \le 1,$ $\phi_j=1$ on $B(0,j)$ and such that $\phi_j$ is compactly supported in $B(0, j+1)$ for every $j \ge 1.$
We recall that in the discussion before the proof of Case 1, we extended the functions $u_k,$ $v_{k,i}$ and $u$ by zero on ${\mathbf{R}}^n \setminus \Omega$ and we denoted these extensions by $\widetilde{u}_k,$ $\widetilde{v}_{k,i}$ and $\widetilde{u}$ respectively. We noticed then that all the functions $\widetilde{u}_{k}$ and $\widetilde{v}_{k,i}$ are continuous on ${\mathbf{R}}^n.$

Let $j \ge 1$ be a fixed integer and let $s \ge 1$ be a finite number. Let $\Omega_j:= \Omega \cap B(0,j+1).$ Via Lemma \ref{Product Rule for W1pq} and
Remark \ref{boundedness product rule for H1pq and W1pq} (see also \cite[Lemma 4.9 and Theorem 4.11]{Cos4}) we have that $w \phi_j$ is in $H_0^{1,(p,s)}(\Omega_j)$ whenever $w$ is in $H_0^{1,(p,s)}(\Omega)$ with
$$||w \phi_j||_{H^{1,(p,s)}(\Omega_j)} \le 3 ||w||_{H^{1,(p,s)}(\Omega)}$$
for all $w \in H_{0}^{1,(p,s)}(\Omega).$

Thus, via Lemma \ref{Product Rule for W1pq} and Remark \ref{boundedness product rule for H1pq and W1pq}
(see also \cite[Lemma 4.9 and Theorem 4.11]{Cos4}), we see that the sequence $u_k \phi_j$ is bounded in $H_{0}^{1,(p,1)}(\Omega_j)$ since $\phi_j \in C_{0}^{\infty}({\mathbf{R}}^n)$ is compactly supported in $B(0,j+1)$ and since the sequence $u_k$ is bounded in $H_{0}^{1,(p,1)}(\Omega).$

It is also easy to see via \cite[Lemma 4.9]{Cos4} that we have $u_k \phi_j \rightarrow u \phi_j$ weakly in $L^{(p,q)}(\Omega_j)$ and $\nabla (u_k \phi_j) \rightarrow \nabla (u \phi_j)$ weakly in $L^{(p,q)}(\Omega_j; {\mathbf{R}}^n)$ since $\phi_j \in C_{0}^{\infty}({\mathbf{R}}^n)$ is compactly supported in $B(0,j+1),$ $u_k \rightarrow u$ weakly in $L^{(p,q)}(\Omega)$ and since $\nabla u_k \rightarrow \nabla u$ weakly in $L^{(p,q)}(\Omega; {\mathbf{R}}^n).$ By applying Case 1 to the sequence $(u_k \phi_j)_{k \ge 1}$ with respect to the bounded open set $\Omega_j,$ we see that $u \phi_j \in H_{0}^{1, (p,1)}(\Omega_j).$ Thus, $u \phi_j \in H_{0}^{1,(p,1)}(\Omega_j) \subset H_{0}^{1,(p,1)}(\Omega)$ for every $j \ge 1$ integer.

By doing a computation similar to the one from the proof of our previous result \cite[Theorem 4.12]{Cos4}, we obtain
\begin{eqnarray*}
||\widetilde{u}- \widetilde{u} \phi_j||_{H^{1,(p,1)}({\mathbf{R}}^n)} &\le&
||\widetilde{u}(1-\phi_j)||_{L^{(p,1)}({\mathbf{R}}^n)}+||\widetilde{u} \nabla \phi_j||_{L^{(p,1)}({\mathbf{R}}^n;{\mathbf{R}}^n)}
 +||(1-\phi_j) \nabla \widetilde{u}||_{L^{(p,1)}({\mathbf{R}}^n;{\mathbf{R}}^n)}\\
&\le& 3 \, ||\widetilde{u} \chi_{{\mathbf{R}}^n \setminus B(0,j)}||_{L^{(p,1)}({\mathbf{R}}^n)}+||\nabla \widetilde{u} \chi_{{\mathbf{R}}^n \setminus B(0,j)}||_{L^{(p,1)}({\mathbf{R}}^n;{\mathbf{R}}^n)}
\rightarrow 0
\end{eqnarray*}
as $j \rightarrow \infty.$ From this, \cite[Proposition 5.2]{Cos4}, the definition of $\widetilde{u}$ and the fact that $\widetilde{u} \phi_j$ is the extension by $0$ on ${\mathbf{R}}^n \setminus \Omega_j$ of $u \phi_j \in H_{0}^{1,(p,1)}(\Omega_j) \subset H_{0}^{1,(p,1)}(\Omega)$ for every $j \ge 1$ integer, it follows that $u \in H_{0}^{1,(p,1)}(\Omega).$ This finishes the proof of the case $\Omega \neq {\mathbf{R}}^n$ unbounded. Thus, we finish proving claim (ii) and the theorem.

\end{proof}

\begin{Remark} \label{Remark on bdd in H01p1 weak limit in H01p1 1<p<n}
When proving this weak convergence result for $H_{0}^{1,(p,1)}(\Omega),$ we relied heavily many times on the fact that we can work with continuous functions from $H_{0}^{1,(p,1)}(\Omega)$ whenever $1 \le n<p<\infty$ or $1<n=p<\infty.$ The existence of discontinuous and/or unbounded functions in $H_{0}^{1,(p,1)}(\Omega)$ when $1<p<n$ leaves as an open question the membership of the limit function $u$ in $H_{0}^{1,(p,1)}(\Omega)$ when $1<p<n,$ $\Omega \subset {\mathbf{R}}^n$ is bounded and $u$ is not compactly supported in $\Omega.$ Thus, we do not know at this point in time whether the weak convergence result concerning $H_{0}^{1,(p,1)}(\Omega)$ can be extended to the case $1<p<n.$
\end{Remark}

The following proposition will be useful in the sequel.

\begin{Proposition} \label{size of the set of pointwise convergence when p=n} Suppose that $1<n,q<\infty,$ where $n$ is an
integer. Let $\Omega \subset {\mathbf{R}}^n$ be an open set. Let $u$ be a function in $C(\Omega)  \cap H_{0}^{1,(n,1)}(\Omega)$ and let $(u_k)_{k \ge 1} \subset C(\Omega) \cap H_{0}^{1,(n,1)}(\Omega)$ be a sequence in $H_{0}^{1,(n,1)}(\Omega)$ such that
$$||u_k-u||_{L^{(n,q)}(\Omega)}+||\nabla u_k-\nabla u||_{L^{(n,q)}(\Omega; {\mathbf{R}^n})} < 2^{-2k}$$
for every $k \ge 1.$ Then there exists a Borel set $F \subset \Omega$ such that ${\mathrm{Cap}}_{n,q}(F)=0$
and such that $u_k \rightarrow u$ pointwise on $\Omega \setminus F.$

\end{Proposition}

\begin{proof} For every $k \ge 1$ let

$$O_{k}=\{ x \in \Omega: |u_{k+1}(x) - u_{k}(x)|>2^{-k} \}  \mbox { and } U_k=\bigcup_{l \ge k} O_l.$$
Since all the functions $u_{j}$ are continuous on $\Omega,$ it follows that $O_{k}$ is in fact an open subset of $\Omega$ for every $k \ge 1.$ For every $k \ge 1,$ the function $w_{k}:=2^{k} |u_{k+1}-u_{k}|$ is admissible for the open set $O_{k}$ with respect to the global $(n,q)$-capacity and we have
\begin{eqnarray*}
{\mathrm{Cap}}_{(n,q)}(O_{k})^{1/n} &\le& ||w_{k}||_{L^{(n,q)}(\Omega)}+||\nabla w_{k}||_{L^{(n,q)}(\Omega; {\mathbf{R}^n})}\\
&=& 2^{k} (||u_{k+1}-u_{k}||_{L^{(n,q)}(\Omega)} + ||\nabla u_{k+1} - \nabla u_{k}||_{L^{(n,q)}(\Omega; {\mathbf{R}^n})})\\
&\le& 2^{k} (||u_{k+1}-u||_{L^{(n,q)}(\Omega)} + ||\nabla u_{k+1} - \nabla u||_{L^{(n,q)}(\Omega; {\mathbf{R}^n})})\\
& & + 2^{k} (||u_{k}-u||_{L^{(n,q)}(\Omega)} + ||\nabla u_{k} - \nabla u||_{L^{(n,q)}(\Omega; {\mathbf{R}^n})})\\
&<& 2^{k} \, (2^{-2(k+1)}+2^{-2k}) <2^{1-k}.
\end{eqnarray*}

The set $U_{k}$ is a countable union of open sets in $\Omega$, hence it an open set in $\Omega$ itself and
$${\mathrm{Cap}}_{n,q}(U_{k})^{1/n} \le {\mathrm{Cap}}_{(n,q)}(U_{k})^{1/n} \le \sum_{j=k}^{\infty} {\mathrm{Cap}}_{(n,q)}(O_{j})^{1/n} \le \sum_{j=k}^{\infty} 2^{1-j}=2^{2-k}.$$
Let $F=\bigcap_{k \ge 1} U_{k}.$ It follows immediately that $F$ is a Borel set and ${\mathrm{Cap}}_{n,q}(F)={\mathrm{Cap}}_{(n,q)}(F)=0.$

Let $v: \Omega \rightarrow {\mathbf{R}}$ be the function
$$v(x)=\left\{ \begin{array}{ll}
\lim_{k \rightarrow \infty} u_{k}(x) & \mbox{ if } x \in \Omega \setminus F\\
0 & \mbox{ if } x \in F.
\end{array}
\right.$$

We notice that $u_{k}$ converges to $v$ pointwise in $\Omega \setminus F$ and uniformly on the sets $\Omega \setminus U_{j}, j \ge 1.$ In particular $v$ is continuous when restricted to the sets $\Omega \setminus U_{j}, j \ge 1.$

We know that $u=v$ almost everywhere in $\Omega$ since the sequence $u_{k}$ converges to $u$ in $H_{0}^{1,(n,q)}(\Omega)$ and to $v$ almost everywhere in $\Omega.$ We claim that $u=v$ on $\Omega \setminus F.$ This would imply that that $v_{k}$ converges to $u$ pointwise in $\Omega \setminus F$ and uniformly on the sets $\Omega \setminus U_{j}, j \ge 1.$

In order to prove that $u=v$ on $\Omega \setminus F,$ it is enough to prove that $u=v$ on $\Omega \setminus U_{j}$ for all $j \ge 1$ since $F=\bigcap_{j \ge 1} U_{j}.$

Let $j \ge 1$ be fixed. We study two separate cases here, depending on whether $\Omega$ is bounded or not.

\smallskip

Case 1. Assume that $\Omega$ is bounded. We can assume without loss of generality via Theorem \ref{continuous embedding of H01n1 into C cap Linfty} that $u$ and the functions $u_k$ are continuous on $\overline{\Omega}$ and $0$ on $\partial \Omega.$ We can also extend $v$ by $0$ on $\partial \Omega.$

Since $u$ is continuous on $\overline{\Omega},$ since $v$ is continuous when restricted to $\Omega \setminus U_{j}$ and since $u=v$ almost everywhere in $\Omega$ we have that $u=v$ pointwise on the open set $\Omega \setminus \overline{U}_{j}$ because all the points in this open set are Lebesgue points for both $u$ and $v.$ We still have to show that $u=v$ on $\partial U_{j}.$ Since the functions $u$ and $v$ agree on $\Omega \setminus \overline{U}_{j}$ and on $\partial \Omega$ and since they are both continuous when restricted to $\Omega \setminus U_{j},$ it follows that they agree on $\partial U_{j}$ as well. Therefore, $u=v$ on $\Omega \setminus U_{j}$ when $\Omega$ is bounded.

\smallskip

Case 2. We assume now that $\Omega$ is unbounded. We can assume without loss of generality that $0 \in \Omega.$ Like in the proof of \cite[Theorem 4.12]{Cos4}, we choose a sequence of 2-Lipschitz smooth functions
$(\phi_m)_{m \ge 1} \subset C_{0}^{\infty}({\mathbf{R}}^n)$ such that $0 \le \phi_m \le 1,$ $\phi_m=1$ on $B(0,m)$ and such that $\phi_m$ is compactly supported in $B(0,m+1)$ for every integer $m \ge 1.$ For a fixed $m \ge 1$ let
$$O_{k,m}=\{ x \in \Omega: |(u_{k+1}\phi_m)(x) - (u_{k}\phi_m)(x)|>2^{-k} \} \mbox { and } U_{k,m}=\bigcup_{l \ge k} O_{l,m}.$$

For every fixed $m \ge 1$ it is easy to see that $u \phi_m \in C(\Omega) \cap H_{0}^{1,(n,1)}(\Omega)$ and that $(u_k \phi_m)_{k \ge 1} \subset C(\Omega) \cap H_{0}^{1,(n,1)}(\Omega).$ Moreover,

\begin{eqnarray*}
||u_k \phi_m-u \phi_m||_{L^{(n,q)}(\Omega)} &\le & ||u_k -u ||_{L^{(n,q)}(\Omega)} \mbox{ and } \\
||\nabla (u_k \phi_m)-\nabla (u \phi_m)||_{L^{(n,q)}(\Omega; {\mathbf{R}}^n)} &\le& ||\nabla u_k-\nabla u||_{L^{(n,q)}(\Omega; {\mathbf{R}}^n)}+ 2 ||u_k-u||_{L^{(n,q)}(\Omega)}
\end{eqnarray*}
for every $k \ge 1.$

From our choice of the sequence of the sequence $(\phi_m)_{m \ge 1}$ it follows that
$$O_k \cap B(0,m) \subset O_{k,m} \subset O_k \cap B(0, m+1) \mbox{ and } U_k \cap B(0,m) \subset U_{k,m} \subset U_k \cap B(0, m+1)$$
for all integers $k, m \ge 1.$

By applying Case 1 to the sequence $(u_k \phi_m)_{k \ge 1}$ and to the bounded sets $U_{j,m} \subset U_j \cap B(0, m+1)$ and $\Omega \cap B(0,m+1),$ we see that $u_k \phi_m \rightarrow u \phi_m$ uniformly on $\Omega \setminus U_{j,m}$ for every $m \ge 1.$ From this, the definition of the functions $\phi_m$ and the fact that $u_k \rightarrow v$ uniformly on $\Omega \setminus U_j,$ it follows that $u \phi_m=v \phi_m$ on $\Omega \setminus U_j$ for every $m \ge 1.$ Thus, $u=v$ on $\Omega \setminus U_j$ when $\Omega$ is unbounded. This finishes the proof of Case 2 and the proof of the proposition.

\end{proof}

\section{Choquet property for the capacities associated to $H_{0}^{1,(p,1)}(\Omega)$}
\label{section Choquet property for the p1 capacities}

In this section we prove that the Sobolev-Lorentz relative and global capacities defined via the $(p,1)$ norm and respectively
via the $p,1$ norm have the Choquet property whenever $1 \le n<p<\infty$ or $1<n=p<\infty.$ We prove that all these set functions satisfy a Monotone Convergence Theorem-type result whenever $1 \le n<p<\infty$ or $1<n=p<\infty.$ See Theorems \ref{Cap Thm (p,q) relative capacity} (v), \ref{Cap Thm p,q relative capacity} (v), \ref{Cap Thm (p,q) global capacity} (iv) and respectively \ref{Cap Thm p,q global capacity} (iv) and the discussions before and after Questions \ref{Question Choquet (p,q) relative capacity}, \ref{Question Choquet p,q relative capacity}, \ref{Question Choquet (p,q) global capacity} and respectively \ref{Question Choquet p,q global capacity}.

We start by showing that the Monotone Convergence Theorem holds for the $(p,1)$ and the $p,1$ relative capacities whenever $1 \le n<p<\infty$ or $1<n=p<\infty.$

\begin{Theorem} \label{MCT for the (p,1) and p,1 relative capacities}
Let $n \ge 1$ be an integer. Suppose that $1 \le n<p<\infty$ or $1<n=p<\infty.$ Let $\Omega \subset {\mathbf{R}}^n$ be bounded and open. Let $E_k$ be an increasing set of subsets in $\Omega$ and let $E=\bigcup_{k=1}^{\infty} E_k.$ Then

\par {{\rm(i)}} $\lim_{k \rightarrow \infty} {\rm{cap}}_{(p,1)}(E_k, \Omega) = {\rm{cap}}_{(p,1)}(E, \Omega)$

\par {{\rm(ii)}} $\lim_{k \rightarrow \infty} {\rm{cap}}_{p,1}(E_k, \Omega) = {\rm{cap}}_{p,1}(E, \Omega).$

\end{Theorem}

\begin{proof} We start by proving claim (i).

Due to the monotonicity of ${\mathrm{cap}}_{(p,1)}(\cdot, \Omega),$ we have obviously
$$L:=\lim_{k \rightarrow \infty} {\rm{cap}}_{(p,1)}(E_k, \Omega)^{1/p} \le {\rm{cap}}_{(p,1)}(E, \Omega)^{1/p}.$$

To prove the opposite inequality, we may assume without loss of generality that $L < \infty.$
Let $\varepsilon \in (0,1)$ be fixed. For every $k \ge 1$ we choose $u_{k} \in
{\mathcal{A}}(E_{k}, \Omega)$ such that $0 \le u_k \le 1$ and
\begin{equation}\label{vv 1}
||\nabla u_k||_{L^{(p,1)}(\Omega; {\mathbf{R}}^n)}
<{\mathrm{cap}}_{(p,1)}(E_k, \Omega)^{1/p}+ \varepsilon
\end{equation}
for every $k \ge 1.$

Via Theorem \ref{continuous embedding of H01n1 into C cap Linfty} when $1<n=p<\infty$, via \cite[Theorem 5.5 (iii)]{Cos4} when $1=n<p<\infty$ or via \cite[Theorem 5.6 (iv)]{Cos4} when $1<n<p<\infty$ we can assume without loss of generality (since $u_k=1$ on an open neighborhood of $E_k$) that $u_k$ is in $C(\overline{\Omega}) \cap H_{0}^{1,(p,1)}(\Omega)$ and zero on $\partial \Omega$ for every $k \ge 1.$

We notice that the sequence $(u_k)_{k \ge 1} \subset C(\overline{\Omega}) \cap H_{0}^{1,(p,1)}(\Omega)$ is bounded in $H_{0}^{1,(p,1)}(\Omega)$ because
the sequence $(u_k, \nabla u_k)_{k \ge 1}$ is bounded in $L^{(p,1)}(\Omega) \times L^{(p,1)}(\Omega; {\mathbf{R}}^n).$

Since $H_{0}^{1,(p,1)}(\Omega) \subset H_{0}^{1,p}(\Omega)$ and the sequence $u_k$ is bounded in $H_{0}^{1,(p,1)}(\Omega),$
it follows that $u_k$ is bounded in the reflexive space $H_{0}^{1,p}(\Omega).$ (See the discussion at the beginning of the proof of Theorem \ref{Bdd in H01p1 weak limit in H01p1}). Thus, via Theorem \ref{HKM93 Thm131}
there exists $u \in H_{0}^{1,p}(\Omega)$ and a subsequence,
which we denote again by $u_k,$ such that $(u_k, \nabla u_k) \rightarrow (u, \nabla u)$
weakly in $L^{p}(\Omega) \times L^{p}(\Omega; \mathbf{R}^n)$ as $k \rightarrow
\infty.$

From Theorem \ref{Bdd in H01p1 weak limit in H01p1} (ii) we can assume that $u$ is in fact in $C(\overline{\Omega}) \cap H_{0}^{1,(p,1)}(\Omega)$ and $u=0$ on $\partial \Omega.$ These assumptions can be made via Theorem \ref{continuous embedding of H01n1 into C cap Linfty} when $1<n=p<\infty$, via \cite[Theorem 5.5 (iii)]{Cos4} when $1=n<p<\infty$ or via \cite[Theorem 5.6 (iv)]{Cos4} when $1<n<p<\infty.$ Moreover, from (\ref{Lp1 norm of nabla u le liminf Lp1 norm of nabla uk}) and (\ref{vv 1}) we also have
$$||\nabla u||_{L^{(p,1)}(\Omega; {\mathbf{R}}^n)} \le \liminf_{k \rightarrow \infty} ||\nabla u_k||_{L^{(p,1)}(\Omega; {\mathbf{R}}^n)} \le L+\varepsilon.$$

We want to show that $u=1$ on $E.$ Let $v_{k,3}$ be the sequence constructed in the proof of Theorem \ref{Bdd in H01p1 weak limit in H01p1} (i) by applying Mazur's Lemma with respect to the sequence $(u_k, \nabla u_k)$ and the space $L^{p}(\Omega) \times L^{p}(\Omega; {\mathbf{R}}^n)$ in order to prove (\ref{Lp1 norm of nabla u le liminf Lp1 norm of nabla uk}).
Since the sets of admissible functions are closed under finite convex combinations and $E_k \nearrow E$ as $k \rightarrow \infty,$ we have that $v_{k,3} \in {\mathcal{A}}(E_k,\Omega)$ for every $k \ge 1.$ In particular, $v_{k,3}=1$ on an open neighborhood of $E_k$ for every $k \ge 1.$

We assume first that $1 \le n<p<\infty.$ By inspecting the proof of Theorem \ref{Bdd in H01p1 weak limit in H01p1} (ii) (the case $\Omega$ bounded), we see that the functions $v_{k,3}$ converge uniformly to $u$ on $\overline{\Omega}$ if $1 \le n<p<\infty.$ Since $v_{j,3}$ is $1$ on $E_k$ whenever $j \ge k \ge 1,$ since the functions $v_{k,3}$ converge uniformly to $u$ on $\overline{\Omega}$ and since $E_k \nearrow E$ as $k \rightarrow \infty,$ it follows that $u=1$ on $E$ when $1 \le n<p<\infty.$ Thus, we proved that $u=1$ on $E$ if $1 \le n<p<\infty.$

Assume now that $1<n=p<\infty.$ By inspecting the proof of Proposition \ref{size of the set of pointwise convergence when p=n}, we see that there exists a Borel set $F \subset \Omega$ such that ${\mathrm{Cap}}_{n}(F)=0$ and such that the sequence $v_{k,3}$ converges to $u$ pointwise on $\overline{\Omega} \setminus F.$ Similarly to the notation from Proposition \ref{size of the set of pointwise convergence when p=n}, $F \subset \Omega$ is defined as
$F:=\cap_{k \ge 1} U_k,$ where $U_k=\cup_{j \ge k} O_j$ and
$$O_k=\{ x \in \Omega: |v_{k+1,3}(x)-v_{k,3}(x)|>2^{-k} \}$$
for every $k \ge 1.$

We see that
$O_{j} \cap E_k= \emptyset$ whenever $j \ge k \ge 1$ because $v_{j,3}=1$ on $E_j \supset E_k$ whenever $j \ge k \ge 1.$
Thus, $U_{k} \cap E_k=\emptyset$ for every $k \ge 1,$ which implies $F \cap E=\emptyset.$ Thus, $v_{k,3}$ converges
to $u$ pointwise on $\overline{\Omega} \setminus F \supset E.$ Since $v_{j,3}=1$ on $E_k$ whenever $j \ge k \ge 1$ and since $E_k \nearrow E$ as $k \rightarrow \infty,$ the pointwise convergence of $v_{k,3}$ to $u$ on $E$ implies that $u=1$ on $E$ when $1<n=p<\infty.$ Thus, we proved that $u=1$ on $E$ if $1<n=p<\infty.$

So far we showed that $u \in C(\overline{\Omega}) \cap H_{0}^{1, (p,1)}(\Omega),$ $u=0$ on $\partial \Omega$ and $u=1$ on $E.$ We see that $\frac{u}{1-\varepsilon} \in {\mathcal{A}}(E, \Omega)$ since $u=1$ on $E.$

Thus, we have
$${\mathrm{cap}}_{(p,1)}(E, \Omega)^{1/p} \le \frac{1}{1-\varepsilon} ||\nabla u||_{L^{(p,1)}(\Omega; {\mathbf{R}}^n)} \le \frac{1}{1-\varepsilon}(L+\varepsilon)$$
for every $\varepsilon \in (0,1).$

By letting $\varepsilon \rightarrow 0,$ we obtain
$${\mathrm{cap}}_{(p,1)}(E, \Omega)^{1/p} \le L=\lim_{k \rightarrow \infty} {\mathrm{cap}}_{(p,1)}(E_k, \Omega)^{1/p} \le {\mathrm{cap}}_{(p,1)}(E,\Omega)^{1/p}.$$
This finishes the proof of the claim (i), namely the case of the $(p,1)$ relative capacity. The proof of claim (ii), namely the case of the $p,1$ relative capacity follows by doing an argument very similar to the argument used in the proof of claim (i). This finishes the proof of the theorem.

\end{proof}

From Theorem \ref{MCT for the (p,1) and p,1 relative capacities} (i) and the discussion before Question \ref{Question Choquet (p,q) relative capacity} it follows that the set function ${\mathrm{cap}}_{(p,1)}(\cdot, \Omega)$ satisfies properties (i), (iv) and (v) of Theorem \ref{Cap Thm (p,q) relative capacity} whenever $1 \le n<p<\infty$ or $1<n=p<\infty.$ Thus, ${\mathrm{cap}}_{(p,1)}(\cdot, \Omega)$ is a Choquet capacity (relative to $\Omega$)
whenever $1 \le n<p<\infty$ or $1<n=p<\infty.$ Like in Theorem \ref{Cap Thm (p,q) relative capacity}, the set $\Omega$ is bounded and open in $\mathbf{R}^n,$ where $n \ge 1$ is an integer. We may invoke an important capacitability theorem of Choquet and state the following result. See Doob \cite[Appendix II]{Doo}.

\begin{Theorem}\label{Choquet (p,1) relative capacity Thm}
Let $\Omega$ be a bounded open set in $\mathbf{R}^n,$ where
$n \ge 1$ is an integer. Suppose that $1 \le n<p<\infty$ or $1<n=p<\infty.$
The set function $E \mapsto {\mathrm{cap}}_{(p,1)}(E, \Omega),$
$E \subset \Omega,$ is a Choquet capacity. In particular, all
Borel subsets (in fact, all analytic) subsets $E$ of
$\Omega$ are capacitable, i.e.
$${\mathrm{cap}}_{(p,1)}(E, \Omega)=\sup \, \{ {\mathrm{cap}}_{(p,1)}(K,
\Omega): K \subset E \mbox{ compact} \}.$$
\end{Theorem}

Similarly, from Theorem \ref{MCT for the (p,1) and p,1 relative capacities} (ii) and the discussion before Question \ref{Question Choquet p,q relative capacity} it follows that the set function ${\mathrm{cap}}_{p,1}(\cdot, \Omega)$ satisfies properties (i), (iv) and (v) of Theorem \ref{Cap Thm p,q relative capacity} whenever $1 \le n<p<\infty$ or $1<n=p<\infty.$ Thus, ${\mathrm{cap}}_{p,1}(\cdot, \Omega)$ is a Choquet capacity (relative to $\Omega$) whenever $1 \le n<p<\infty$ or $1<n=p<\infty.$ Like in Theorem \ref{Cap Thm p,q relative capacity}, the set $\Omega$ is bounded and open in $\mathbf{R}^n,$ where $n \ge 1$ is an integer. We may invoke an important capacitability theorem of Choquet and state the following result. See Doob \cite[Appendix II]{Doo}.

\begin{Theorem}\label{Choquet p,1 relative capacity Thm}
Let $\Omega$ be a bounded open set in $\mathbf{R}^n,$ where
$n \ge 1$ is an integer. Suppose that $1 \le n<p<\infty$ or $1<n=p<\infty.$
The set function $E \mapsto {\mathrm{cap}}_{p,1}(E, \Omega),$
$E \subset \Omega,$ is a Choquet capacity. In particular, all
Borel subsets (in fact, all analytic) subsets $E$ of
$\Omega$ are capacitable, i.e.
$${\mathrm{cap}}_{p,1}(E, \Omega)=\sup \, \{ {\mathrm{cap}}_{p,1}(K,
\Omega): K \subset E \mbox{ compact} \}.$$
\end{Theorem}

Now we prove that the Monotone Convergence Theorem holds for the $(p,1)$ and the $p,1$ global capacities whenever $1 \le n<p<\infty$ or $1<n=p<\infty.$

\begin{Theorem} \label{MCT for the (p,1) and p,1 global capacities}
Let $n \ge 1$ be an integer. Suppose that $1 \le n<p<\infty$ or $1<n=p<\infty.$ Let $E_k$ be an increasing set of subsets in ${\mathbf{R}}^n$ and let $E=\bigcup_{k=1}^{\infty} E_k.$ Then

\par{\rm(i)} $\lim_{k \rightarrow \infty} {\rm{Cap}}_{(p,1)}(E_k) ={\rm{Cap}}_{(p,1)}(E)$

\par{\rm(ii)} $\lim_{k \rightarrow \infty} {\rm{Cap}}_{p,1}(E_k) ={\rm{Cap}}_{p,1}(E).$

\end{Theorem}

\begin{proof} We prove the claim in the case of the global $(p,1)$-capacity.

Due to the monotonicity of ${\mathrm{Cap}}_{(p,1)}(\cdot),$ we have obviously
$$L:=\lim_{k \rightarrow \infty} {\rm{Cap}}_{(p,1)}(E_k)^{1/p} \le {\rm{Cap}}_{(p,1)}(E)^{1/p}.$$

To prove the opposite inequality, we may assume without loss of
generality that $L < \infty.$

Let $\varepsilon \in (0,1)$ be fixed. For every $k \ge 1$ we choose $u_k \in
{\mathcal{A}}(E_k)$ such that $0 \le u_k \le 1$ and

\begin{equation}\label{vv 2}
||u_k||_{H^{1,(p,1)}({\mathbf{R}}^n)} <{\mathrm{Cap}}_{(p,1)}(E_k)^{1/p}+ \varepsilon
\end{equation}
for every $k \ge 1.$

Via Theorem \ref{continuous embedding of H01n1 into C cap Linfty} when $1<n=p<\infty,$ via \cite[Theorem 5.5 (iii)]{Cos4} when $1=n<p<\infty$ or via \cite[Theorem 5.6 (iv)]{Cos4} when $1<n<p<\infty$ we can assume without loss of generality (since $u_k=1$ on an open neighborhood of $E_k$) that $u_k$ is in $C({\mathbf{R}}^n) \cap H_{0}^{1,(p,1)}({\mathbf{R}}^n)$ for every $k \ge 1.$

We notice that the sequence $(u_k)_{k \ge 1} \subset C({\mathbf{R}}^n) \cap H_{0}^{1,(p,1)}({\mathbf{R}}^n)$ is bounded in $H_{0}^{1,(p,1)}({\mathbf{R}}^n)$ because
the sequence $(u_k, \nabla u_k)_{k \ge 1}$ is bounded in $L^{(p,1)}({\mathbf{R}}^n) \times L^{(p,1)}({\mathbf{R}}^n; {\mathbf{R}}^n).$

Since $H_{0}^{1,(p,1)}({\mathbf{R}}^n) \subset H_{0}^{1,p}({\mathbf{R}}^n)$ and the sequence $u_k$ is bounded in $H_{0}^{1,(p,1)}({\mathbf{R}}^n),$
it follows that $u_k$ is bounded in the reflexive space $H_{0}^{1,p}({\mathbf{R}}^n).$ (See the discussion at the beginning of the proof of Theorem \ref{Bdd in H01p1 weak limit in H01p1}). Thus, via Theorem \ref{HKM93 Thm131}
there exists $u \in H_{0}^{1,p}({\mathbf{R}}^n)$ and a subsequence,
which we denote again by $u_k,$ such that $(u_k, \nabla u_k) \rightarrow (u, \nabla u)$
weakly in $L^{p}({\mathbf{R}}^n) \times L^{p}({\mathbf{R}}^n; {\mathbf{R}}^n)$ as $k \rightarrow
\infty.$

From Theorem \ref{Bdd in H01p1 weak limit in H01p1} (ii) we can assume that $u$ is in fact in $C({\mathbf{R}}^n) \cap H_{0}^{1,(p,1)}({\mathbf{R}}^n).$ These assumptions can be made via Theorem \ref{continuous embedding of H01n1 into C cap Linfty} when $1<n=p<\infty$, via \cite[Theorem 5.5 (iii)]{Cos4} when $1=n<p<\infty$ or via \cite[Theorem 5.6 (iv)]{Cos4} when $1<n<p<\infty.$ Moreover, from (\ref{H1p1 norm of u le liminf H1p1 norm of uk}) and (\ref{vv 2}) we also have
$$||u||_{H^{1,(p,1)}({\mathbf{R}}^n)} \le \liminf_{k \rightarrow \infty} ||u_k||_{H^{1,(p,1)}({\mathbf{R}}^n)} \le L+\varepsilon.
$$

We want to show that $u=1$ on $E.$ Let $v_{k,1}$ be the sequence constructed in the proof of Theorem \ref{Bdd in H01p1 weak limit in H01p1} (i) by applying Mazur's Lemma with respect to the sequence $(u_k, \nabla u_k)$ and the space $L^{p}({\mathbf{R}}^n) \times L^{p}({\mathbf{R}}^n; {\mathbf{R}}^n)$ in order to prove (\ref{H1p1 norm of u le liminf H1p1 norm of uk}).
Since the sets of admissible functions are closed under finite convex combinations and $E_k \nearrow E$ as $k \rightarrow \infty,$ we have that $v_{k,1} \in {\mathcal{A}}(E_k)$ for every $k \ge 1.$ In particular, $v_{k,1}=1$ on an open neighborhood of $E_k$ for every $k \ge 1.$

We assume first that $1 \le n<p<\infty.$ By inspecting the proof of Theorem \ref{Bdd in H01p1 weak limit in H01p1} (ii) (the case $\Omega={\mathbf{R}}^n$), we see that the functions $v_{k,1}$ converge uniformly to $u$ on compact subsets of ${\mathbf{R}}^n$ if $1 \le n<p<\infty.$ Since $v_{j,1}$ is $1$ on $E_k$ whenever $j \ge k \ge 1,$ since the functions $v_{k,1}$ converge uniformly to $u$ on compact subsets of ${\mathbf{R}}^n$ and since $E_k \nearrow E$ as $k \rightarrow \infty,$ it follows that $u=1$ on $E$ when $1 \le n<p<\infty.$ Thus, we proved that $u=1$ on $E$ if $1 \le n<p<\infty.$

Assume now that $1<n=p<\infty.$ By inspecting the proof of Proposition \ref{size of the set of pointwise convergence when p=n}, we see that there exists a Borel set $F \subset {\mathbf{R}}^n$ such that ${\mathrm{Cap}}_{n}(F)=0$ and such that the sequence $v_{k,1}$ converge to $u$ pointwise on ${\mathbf{R}}^n \setminus F.$ Similarly to the notation from Proposition \ref{size of the set of pointwise convergence when p=n}, $F \subset {\mathbf{R}}^n$ is defined as $F:=\cap_{k \ge 1} U_k,$ where $U_k=\cup_{j \ge k} O_j$ and $O_k=\{ x \in {\mathbf{R}}^n: |v_{k+1,1}(x)-v_{k,1}(x)|>2^{-k} \}.$

We see that $O_{j} \cap E_k= \emptyset$ whenever $j \ge k \ge 1$ because $v_{j,1}=1$ on $E_j \supset E_k$ whenever $j \ge k \ge 1.$
Thus, $U_{k} \cap E_k=\emptyset$ for every $k \ge 1,$ which implies $F \cap E=\emptyset.$ Thus, $v_{k,1}$ converges
to $u$ pointwise on ${\mathbf{R}}^n \setminus F \supset E.$ Since $v_{j,1}=1$ on $E_k$ whenever $j \ge k \ge 1$ and since $E_k \nearrow E$ as $k \rightarrow \infty,$ the pointwise convergence of $v_{k,1}$ to $u$ on $E$ implies that $u=1$ on $E$ when $1<n=p<\infty.$ Thus, we proved that $u=1$ on $E$ if $1<n=p<\infty.$

So far we showed that $u \in C({\mathbf{R}}^n) \cap H_{0}^{1, (p,1)}({\mathbf{R}}^n)$ and $u=1$ on $E.$ We see that $\frac{u}{1-\varepsilon} \in {\mathcal{A}}(E)$ since $u=1$ on $E.$

Thus, we have
$${\mathrm{Cap}}_{(p,1)}(E)^{1/p} \le \frac{1}{1-\varepsilon} ||u||_{H^{1, (p,1)}(\Omega)} \le \frac{1}{1-\varepsilon}(L+\varepsilon)$$
for every $\varepsilon \in (0,1).$

By letting $\varepsilon \rightarrow 0,$ we obtain
$${\mathrm{Cap}}_{(p,1)}(E)^{1/p} \le L=\lim_{k \rightarrow \infty} {\mathrm{Cap}}_{(p,1)}(E_k)^{1/p} \le {\mathrm{Cap}}_{(p,1)}(E)^{1/p}.$$

This finishes the proof of claim (i), namely the case of the $(p,1)$ global capacity. The proof of claim (ii), namely the case of the global $p,1$-capacity follows by doing an argument very similar to the argument used in the proof of claim (i). This finishes the proof of the theorem.

\end{proof}

From Theorem \ref{MCT for the (p,1) and p,1 global capacities} (i) and the discussion before Question \ref{Question Choquet (p,q) global capacity} it follows that the set function ${\mathrm{Cap}}_{(p,1)}(\cdot)$ satisfies properties (i), (iii) and (iv) of Theorem \ref{Cap Thm (p,q) global capacity} whenever $1 \le n<p<\infty$ or $1<n=p<\infty.$ Thus, ${\mathrm{Cap}}_{(p,1)}(\cdot)$ is a Choquet capacity when $1 \le n<p<\infty$ or when $1<n=p<\infty.$
Like in Theorem \ref{Cap Thm (p,q) global capacity}, $n \ge 1$ is an integer. We may invoke an important capacitability theorem of Choquet and state the following result. See Doob \cite[Appendix II]{Doo}.

\begin{Theorem}\label{Choquet (p,1) global capacity Thm}
Let $n \ge 1$ be an integer. Suppose that $1 \le n<p<\infty$ or $1<n=p<\infty.$
The set function $E \mapsto {\mathrm{cap}}_{(p,1)}(E),$
$E \subset {\mathbf{R}}^n,$ is a Choquet capacity. In particular, all
Borel subsets (in fact, all analytic) subsets $E$ of
${\mathbf{R}}^n$ are capacitable, i.e.
$${\mathrm{Cap}}_{(p,1)}(E)=\sup \, \{ {\mathrm{Cap}}_{(p,1)}(K): K \subset E \mbox{ compact} \}.$$
\end{Theorem}

Similarly, from Theorem \ref{MCT for the (p,1) and p,1 global capacities} (ii) and the discussion before Question \ref{Question Choquet p,q global capacity} it follows that the set function ${\mathrm{Cap}}_{p,1}(\cdot)$ satisfies properties (i), (iii) and (iv) of Theorem \ref{Cap Thm p,q global capacity} whenever $1 \le n<p<\infty$ or $1<n=p<\infty.$ Thus, ${\mathrm{Cap}}_{p,1}(\cdot)$ is a Choquet capacity whenever $1 \le n<p<\infty$ or $1<n=p<\infty.$ Like in Theorem \ref{Cap Thm p,q global capacity}, $n \ge 1$ is an integer. We may invoke an important capacitability theorem of Choquet and state the following result. See Doob \cite[Appendix II]{Doo}.

\begin{Theorem}\label{Choquet p,1 global capacity Thm}
Let $n \ge 1$ be an integer. Suppose that $1 \le n<p<\infty$ or $1<n=p<\infty.$
The set function $E \mapsto {\mathrm{Cap}}_{p,1}(E),$
$E \subset {\mathbf{R}}^n,$ is a Choquet capacity. In particular, all
Borel subsets (in fact, all analytic) subsets $E$ of
${\mathbf{R}}^n$ are capacitable, i.e.
$${\mathrm{Cap}}_{p,1}(E)=\sup \, \{ {\mathrm{cap}}_{p,1}(K): K \subset E \mbox{ compact} \}.$$
\end{Theorem}

\vspace{2mm}

\noindent {\bf{Acknowledgements.}} I started writing this article towards the end of my stay
at the University of Pisa and I finished it after my return to the University of Pite\c sti.

\end{document}